\documentclass[12 pt, a4paper]{article}
\usepackage[utf8]{inputenc}
\usepackage{verbatim}
\usepackage[initials]{amsrefs}
\usepackage{authblk}
\usepackage[pdftex]{graphicx}
\usepackage{amsmath}
\usepackage{mathtools}
\usepackage{amsthm}
\usepackage{amsfonts}
\usepackage{amssymb}
\usepackage[margin=1.0in]{geometry}
\usepackage{enumitem}
\usepackage[dvipsnames]{xcolor}
\colorlet{mylinkcolor}{MidnightBlue}
\usepackage{tikz}
\usetikzlibrary{cd}
\usepackage{hyperref}
\hypersetup{
	citecolor = mylinkcolor,
  linkcolor  = mylinkcolor,
  urlcolor   = mylinkcolor,
  colorlinks = true,
}

\usepackage{cleveref}

\newtheorem{theorem}{Theorem}[section]
\newtheorem{cor}[theorem]{Corollary}
\newtheorem{lemma}[theorem]{Lemma}
\newtheorem{prop}[theorem]{Proposition}

\newtheorem*{theorem*}{Theorem}

\theoremstyle{definition}
\newtheorem{defin}[theorem]{Definition}
\newtheorem{fact}[theorem]{Fact}
\newtheorem{exa}[theorem]{Example}
\newtheorem{que}[theorem]{Question}

\theoremstyle{remark}
\newtheorem*{rem}{Remark}
\newtheorem*{claim}{Claim}

\DeclarePairedDelimiter{\set}{\{}{\}}

\DeclarePairedDelimiterX\setnew[2]{\{}{\}}{#1 \nonscript\, \colon \nonscript\, #2}
\newcommand{\onx}[1]{\widehat{#1}}
\newcommand{\interior}{\mathrm{int}}
\newcommand{\sym}{\mathrm{Sym}}

\newcommand{\age}{\mathrm{Age}}

\newcommand{\fr}{Fra\"iss\'e }

\renewcommand{\phi}{\varphi}
\newcommand{\emb}{\mathrm{Emb}}
\newcommand{\aut}{\mathrm{Aut}}
\newcommand{\homeo}{\mathrm{Homeo}}

\newcommand{\ran}{\mathrm{ran}}
\newcommand{\im}{\mathrm{Im}}

\newcommand{\hatf}{\,\hat{\rule{-0.5ex}{1.5ex}\smash{f}}}

\renewcommand{\b}[1]{\mathbf{#1}}
\renewcommand{\cal}[1]{\mathcal{#1}}
\newcommand{\bb}[1]{\mathbb{#1}}

\newcommand{\ap}{\mathrm{AP}}
\newcommand{\lip}{\mathrm{Lip}}
\renewcommand{\S}{\mathrm{S}}
\newcommand{\M}{\mathrm{M}}

\def\-{\raisebox{.30pt}{-}}

\begin{document}

	\title{Topological dynamics beyond Polish groups}
	\author{Gianluca Basso and Andy Zucker}

\date{}

	\maketitle

	\begin{abstract}
		When $G$ is a Polish group, metrizability of the universal minimal flow has been shown to be a robust dividing line in the complexity of the topological dynamics of $G$. We introduce a class of groups, the CAP groups, which provides a neat generalization of this to all topological groups. We prove a number of characterizations of this class, having very different flavors, and use these to prove that the class of CAP groups enjoys a number of nice closure properties. As a concrete application, we compute the universal minimal flow of the homeomorphism groups of several scattered topological spaces, building on recent work of Gheysens.
		\let\thefootnote\relax
		\let\thefootnote\relax\footnote{The first author's work was partially conducted within the program ``Investissements d'Avenir'' (ANR-16-IDEX-0005) operated by the French National Research Agency (ANR). The second author was supported by NSF Grant no.\ DMS 1803489.}
	\end{abstract}

	\section{Introduction}

	This paper is a contribution to the study of abstract topological dynamics; see Subsection~\ref{SubSec:TopologicalDynamics} for definitions. A classical theorem of Ellis \cite{Ellis} shows that every topological group $G$ admits a \emph{universal minimal flow}, or \emph{UMF}. This is a minimal flow which admits a $G$-map onto any other minimal flow. It is unique up to isomorphism and denoted $\M(G)$. Often, $\M(G)$ is extremely large, for instance when $G$ is an infinite discrete group. But there are examples of \emph{extremely amenable} topological groups, i.e.\ groups $G$ where $\M(G)$ is a singleton, as is the case for the group $\aut(\bb Q)$ of order-preserving bijections of the rationals under the topology of pointwise convergence \cite{Pes}. Other times, $\M(G)$ is non-trivial, but still metrizable and easy to describe; as an example, when $G = \sym(\omega)$, the group of permutations of $\omega$ with the pointwise convergence topology, then we have $\M(G) \cong \mathrm{LO}(\omega)$, the space of linear orders on $\omega$ \cite{GW}.

	Recall that a topological group is \emph{Polish} if its underlying topological space is Polish, i.e.\ separable and admitting a compatible, complete metric. Metrizabilty of the universal minimal flow has emerged as a meaningful dividing line in the topological dynamics of Polish groups. Starting with the seminal paper of Kechris, Pestov, and Todor\v{c}evi\v{c} \cite{KPT} and with further work by Melleray-Nguyen Van Th\'e-Tsankov \cite{MNT}, Zucker \cite{ZucAut}, and Ben Yaacov-Melleray-Tsankov \cite{BYMT}, the structure of $\M(G)$ is more-or-less completely understood when $G$ is Polish and $\M(G)$ is metrizable. In this case, one can find a closed extremely amenable subgroup $H\subseteq G$ so that $\M(G)\cong \onx{G/H}$, the right completion of the space of left cosets. When $G = \aut(\b{K})$ for some countable, first order structure $\b{K}$, this can be given a combinatorial interpretation. For instance, when $G = \sym(\omega)$, one can let $H = \aut(\langle \omega, \preceq\rangle)$, where $\preceq$ is some dense linear order without endpoints on $\omega$, and we have $\widehat{G/H} \cong \mathrm{LO}(\omega)$.

	However, the metrizability of $\M(G)$ stops being a relevant indicator for ``nice dynamics" when $G$ is not Polish. Indeed, the work by Barto\v{s}ov\'a \cite{Bar} on groups of automorphisms of uncountable structures shows that $\M(\aut(\b K))$ may have a concrete representation while being far from metrizable. As an example, for any cardinal $\kappa$, one can form the group $\sym(\kappa)$ of permutations of $\kappa$, again with the pointwise convergence topology. Then, as also noted in \cite{Pes2}, $\M(\sym(\kappa)) = \mathrm{LO}(\kappa)$, neatly generalizing the result of \cite{GW}.

	In this paper we generalize and extend several of the aforementioned results and provide a framework for understanding and classifying the dynamics of general topological groups. Most notably, we isolate a notion which coincides with metrizability of the UMF for Polish groups and captures when the group has nice dynamics. We prove that this is a robust notion: the criterion is equivalent to a variety of statements, which come in different flavors, and the class of groups which satisfy it is well behaved and enjoys strong closure properties.

	Given a $G$-flow $X$, one can look at the collection of points which belong to minimal subflows of $X$.
	These points are called \emph{almost periodic} and are denoted by $\ap_G(X)$.
	While $\ap_G(X)$ is clearly invariant under the action of $G$, it is not in general a subflow, because it might not be closed.
	We say that a topological group $G$ has \emph{Closed AP}, or is \emph{CAP}, if $\ap_G(X)$ is closed for each $G$-flow $X$.
	It is easy to see that all pre-compact groups are CAP and that no locally compact non-compact groups are CAP.
	The results of Barto\v{s}ov\'a-Zucker, appearing in \cite{ZucThesis}, and Jahel-Zucker in \cite{JZ} show that for $G$ Polish, $\M(G)$ is metrizable if and only if $G$ is CAP. We show that this notion is still relevant also for groups beyond Polish, unlike metrizabilty of the UMF.
	In general, it is open whether the class of CAP groups can be characterized by the topology of $\M(G)$ alone; see \Cref{Que:TopologyMGCAP}.

	In \cite{BYMT} it is shown that for each Polish group $G$, one can endow $\M(G)$ with a metric which is in general finer than the compact topology, but interacts with the compact topology in non trivial ways: together they form a \emph{topo-metric} space. They show that this metric is compatible exactly when $\M(G)$ is metrizable. It is shown in \cite{ZucMHP} that this finer metric on $\M(G)$ is entirely canonical, and does not depend on various choices made during the construction. In this work, we introduce \emph{topo-uniform} spaces, which are sets endowed with a topology and a uniformity which, while not generally compatible with the topology, interacts with it in key ways. We describe a canonical uniform structure on $\M(G)$, the \emph{UEB uniformity}, and show that together with the usual compact topology it forms a topo-uniform space. One of our main results is the following, contained in \Cref{Thm:EquivalentInG}.
	\vspace{2 mm}

	\begin{theorem*}
		Let $G$ be a topological group. Then the following are equivalent.
		\begin{enumerate}
			\item
			$G$ is CAP.
			\item
			The UEB uniformity and the compact uniformity coincide on $\M(G)$.
		\end{enumerate}
	\end{theorem*}
\vspace{2 mm}

	Having both characterizations of CAP groups to work with allows us to prove a variety of properties about them. The class of CAP groups is closed under quotients, group extensions, surjective inverse limits and arbitrary products.
	Generalizing results in \cite{MNT}, we find sufficient conditions for being CAP in terms of the existence of well behaved subgroups, conditions which are known to be necessary when $G$ is Polish; we do not know whether these are necessary in general. We also prove the following peculiar characterization of CAP groups, \Cref{Cor:CAPviaProducts}, in terms of whether the UMF respects product.
	\vspace{2 mm}

\begin{theorem*}
	Let $G$ be a topological group. Then the following are equivalent:
	\begin{enumerate}
		\item
		$G$ is CAP.
		\item
		$\M(G\times G)\cong \M(G)\times \M(G)$.
	\end{enumerate}
\end{theorem*}
\vspace{2 mm}

More generally, we show that the UMF of an arbitrary product of CAP groups is simply the product of the UMFs.

One source of interesting CAP groups comes from considering the automorphism groups of $\omega$-homogeneous structures; see Section~\ref{Sec:Structures} for the definitions. The first systematic study of the automorphism groups of uncountable $\omega$-homogeneous structures is the work of Barto\v{s}ov\'a \cites{BarMore, Bar, Bar2}, where similar criteria as those developed in \cite{KPT} are used to compute the universal minimal flows of several such automorphism groups. In the countable case, it is shown in \cite{ZucAut} that for the automorphism group of a countable $\omega$-homogeneous structure, the combinatorial property of having \emph{finite Ramsey degrees} characterizes when the UMF is metrizable. We generalize this to the uncountable setting in \Cref{Thm:FiniteRamseyIffCAP}.
\vspace{2 mm}

\begin{theorem*}
	Let $\b K$ be a $\omega$-homogeneous relational structure.
	Then $\aut(\b K)$ is CAP if and only if  $\age(\b K)$ has finite Ramsey degrees.
\end{theorem*}
\vspace{2 mm}

Using this and Barto\v{s}ov\'a's results, we compute $\M(\homeo(\omega_1))$, as well as the UMFs of the groups of homeomorphisms of a variety of scattered topological spaces, building on recent work of Gheysens \cite{Ghey}.

The paper is organized as follows. Section~\ref{Sec:Prelim} contains several preliminary results on topological groups and dynamics, uniform spaces, and the Samuel compactification $\S(G)$ that we will need going forward. Section~\ref{Sec:CAP} gives a brief introduction to the class of CAP groups. Section~\ref{Sec:TopoUnif} introduces the UEB uniformity on $\S(G)$ and $\M(G)$ and shows that this uniformity, while in general finer than the compact topology, interacts with it in nice ways. Section~\ref{Sec:UEB} introduces the class of \emph{UEB groups}, those groups where the compact and UEB uniformities agree on $\M(G)$. While we show in Section~\ref{Sec:EquivCAPUEB} that the classes of CAP groups and UEB groups coincide, Section~\ref{Sec:UEB} collects several results about this class which are easier to prove using the UEB characterization. Sections~\ref{Sec:ClosureProps} and \ref{Sec:LargeCAPSubs} give a variety of sufficient conditions under which a group is CAP, and show that the class of CAP group enjoys nice closure properties. Lastly, Section~\ref{Sec:Structures} investigates the automorphism groups of $\omega$-homogeneous structures and the homeomorphism groups of scattered spaces.

	\subsection*{Notation}

	Our notation is mostly standard. We let $\omega := \{0, 1, 2, ...\}$ denote the least infinite ordinal. Whenever $G$ is a topological group, we write $1_{G}$ for its identity and $\cal{N}_G$ for a base of symmetric open neighborhoods of $1_G$; if $d$ is a continuous, right-invariant pseudometric on $G$ and $c> 0$, we set $d(c) := \{g\in G: d(1_G, g)< c\}$. If $X$ is a topological space and $x\in X$, we often write $A \ni_{op} x$ or $A\subseteq_{op} X$ to introduce a non-empty open set $A$ in $X$.

	\subsection*{Acknowledgments}

	We would like to thank Colin Jahel for numerous detailed discussions during the early stages of the project. We would also like to thank Jan Pachl for helpful discussions about the UEB uniformity.

	\section{Preliminaries on topological groups}
	\label{Sec:Prelim}

	We collect some preliminaries on topological groups, with a particular focus on topological dynamics. We emphasize that all groups and spaces appearing in this paper are Hausdorff.

	\subsection{The Birkhoff-Kakutani theorem}
	\label{SubSec:BK}

	The famous theorem of Birkhoff and Kakutani states that a topological group $G$ is metrizable if and only if it is first countable. However, the proof of the theorem actually says something non-trivial about all topological groups, regardless of whether or not they are first countable.

	Recall that a \emph{pseudo-metric} on a set $X$ is a function $d\colon X\times X\to \bb{R}^{\geq 0}$ which satisfies each of the conditions for being a metric except possibly that of distinguishing distinct points.
	A pseudo-metric $d$ on a group $G$ is \emph{right-invariant} if $d(g, g') = d(gh, g'h)$ for all $g, g', h \in G$.
	\vspace{2mm}

	\begin{fact}[\cite{Ber}*{p.~28}]
		\label{Fact:BK}
		Suppose $G$ is a topological group, and let $\{U_n: n< \omega\}\subseteq \cal{N}_G$. Then there is a continuous, bounded, right-invariant pseudo-metric $d$ on $G$ so that for every $n< \omega$, there is $\epsilon_n> 0$ with $\{g: d(1_G, g)< \epsilon_n\}\subseteq U_n$.
	\end{fact}
	\vspace{2 mm}

	Continuous, right-invariant pseudo-metrics on $G$ will feature prominently throughout the paper, so we fix some notation. If $d$ is a continuous, right-invariant pseudo-metric on $G$ and $c > 0$, we write $d(c) \coloneqq \{g\in G: d(1_G, g)< c\}$. Let $\lip(d)$ denote the collection of functions from $G$ to $[0,1]$ which are $1$-Lipschitz with respect to $d$, that is, such that $\lvert f(g) - f(h) \rvert \le d(g, h)$, for all $g, h \in G$.

	We say that a collection $\cal{D}$ of continuous, diameter $1$, right-invariant pseudo-metrics is a \emph{strong base of pseudo-metrics} if $\{d(1): d\in \cal{D}\}$ is a neighborhood basis at $1_G$. We say that $\cal{D}$ is a \emph{base of pseudo-metrics} if $\{d(\epsilon): d\in \cal{D}, \epsilon > 0\}$ is a neighborhood basis at $1_G$. By \Cref{Fact:BK}, every topological group admits a base of bounded pseudo-metrics. By taking a base of pseudo-metrics, multiplying each member by some constants, and capping at $1$, one can obtain a strong base of pseudo-metrics.

	\subsection{Uniform spaces}
	\label{SubSec:UnifSpaces}

	Let $X$ be a set.
	For $U, V \subseteq X \times X$, we write $U^{-1} \coloneqq \{(y, x):  (x, y)\in U \}$ and
	$$UV \coloneqq \{(x, y): \exists z\in X\, (x, z)\in U \text{ and } (z, y)\in V\},$$
	and likewise for the ``product" of any finitely many subsets of $X\times X$.

	A (Hausdorff) \emph{uniform space} is a set $X$ together with a filter $\mathcal{U}$ of supersets of the diagonal $\Delta \subseteq X \times X$ such that:
\begin{itemize}
	\item
		for each $U \in \mathcal{U}$ there is $V \in \mathcal{U}$ with $V^2  \subseteq U$,
	\item
		if $U \in \mathcal U$, then $U^{-1} \in \mathcal{U}$,
	\item
		$\bigcap_{U \in \mathcal{U}} U = \Delta$.
\end{itemize}
Members of $\cal{U}$ are called \emph{entourages}.

 For $U, V\in \cal{U}$, we write $V\ll U$ if for some $W\in \cal{U}$, we have $WVW\subseteq U$. For $A \subseteq X$ and $U \in \mathcal{U}$ we write:
\begin{align*}
A[U] &\coloneqq \setnew{y \in X}{\exists x \in A \, (x, y) \in U}\\[1 mm]
A(U) &\coloneqq \bigcup_{V\ll U} A[V]
\end{align*}

When $A = \set{x}$, we write $x[U], x(U)$ in place of $\set{x}[U], \set{x}(U)$.

The \emph{uniform topology} on $X$ is given by declaring a set $A$ to be open if for each $x \in A$ there is $U \in \cal{U}$ with $x[U] \subseteq A$.
We say that a topology $\tau$ on $X$ is compatible with the uniform structure if $\tau$ coincides with the uniform topology.
Each compact space admits a unique uniform structure; its entourages are all the neighborhoods of the diagonal.

A function $f\colon X \to Y$ between uniform spaces is \emph{uniformly continuous} if for each entourage $V$ of $Y$ there is an entourage $U$ of $X$ such that $(f(x), f(y)) \in V$ for all $(x, y) \in U$.

The \emph{right uniformity} on a topological group $G$ is the uniformity generated by all continuous right-invariant pseudo-metrics on $G$. By \Cref{Fact:BK}, this is equivalent to the uniformity whose typical basic entourage is of the form $\{(g, h)\in G\times G: gh^{-1}\in U\}$ for some $U\in \cal{N}_G$. If $H\subseteq G$ is a closed subgroup, then the left coset space $G/H$ also has a natural right uniformity whose typical basic entourage is of the form $\{(gH, kH): gHk^{-1}\cap U\neq\emptyset\}$ for some $U\in \cal{N}_G$.

Every uniform space admits $X$ admits a \emph{Samuel compactification} $\S(X)$; this is the largest compactification of $X$ with the property that continuous real-valued functions on $\S(X)$ restrict to uniformly continuous functions on $X$. In the case of a topological group $G$ equipped with its right uniform structure, we will see several constructions of $\S(G)$ in Subsection~\ref{SubSec:Samuel}. On occasion, we will also make use of $\S(G/H)$.

\subsection{Topological dynamics}
\label{SubSec:TopologicalDynamics}

Let $G$ denote a Hausdorff topological group. A \emph{$G$-flow} is a compact Hausdorff space $X$ equipped with a continuous (left) action $a\colon G\times X\to X$. Usually the action $a$ is understood, and we simply write $g\cdot x$ or $gx$ in place of $a(g, x)$. If $X$ and $Y$ are $G$-flows, a \emph{$G$-map} is a continuous map $\pi\colon X\to Y$ which respects the $G$-actions.

If $X$ is a $G$-flow and $x\in X$, we define $\rho_x\colon G\to X$ via $\rho_x(g) = gx$.
\vspace{2 mm}

	\begin{fact}
		\label{Fact:ActionUnifCont}
		Suppose $X$ is a $G$-flow. Then for any $x\in X$, the map $\rho_x$ is right uniformly continuous.
	\end{fact}
	\vspace{2 mm}

A corollary of the fact above is the following.
\vspace{2 mm}

	\begin{fact}
		\label{Fact:DenseSubgroup}
		Suppose $H\subseteq G$ is a dense subgroup and that $X$ is an $H$-flow. Then the action continuously extends to $G$.
	\end{fact}
\vspace{2 mm}

A \emph{subflow} of $X$ is a non-empty, closed, $G$-invariant subspace. The $G$-flow $X$ is \emph{minimal} if $X$ contains no proper subflows; equivalently, $X$ is minimal if every orbit is dense. If $\phi\colon X \to Y$ is a $G$-map and $Y$ is minimal then $\phi$ is surjective.
	\vspace{2 mm}

\begin{fact}
	\label{Fact:UMF}
	There exists a \emph{universal minimal flow}, a minimal flow which admits a $G$-map onto any other minimal flow. This flow is unique up to isomorphism and denoted by $\M(G)$.
\end{fact}
\vspace{2 mm}

We will briefly sketch the proof of \Cref{Fact:UMF} in Subsection~\ref{SubSec:Samuel}. Note that \Cref{Fact:DenseSubgroup} implies that if $H\subseteq G$ is a dense subgroup, then $\M(H)\cong \M(G)$ as $H$-flows.

\subsection{The Samuel compactification}
\label{SubSec:Samuel}

Throughout this subsection, fix a topological group $G$. We define an important universal $G$-flow, the \emph{Samuel compactification} of $G$, and present several constructions which we exploit at different points in the paper.
\vspace{2 mm}

\begin{defin}
	\label{Def:SamuelComp}
	We let $\mathrm{RUC}_b(G)$ denote the $C^*$-algebra of bounded right-uniformly continuous functions from $G$ to $\bb{C}$. The \emph{Samuel compactification} of $G$, denoted $\S(G)$, is the Gelfand space of $\mathrm{RUC}_b(G)$, i.e.\ the space of $C^*$-homomorphisms from $\mathrm{RUC}_b(G)$ to $\bb{C}$ endowed with the topology of pointwise convergence.
\end{defin}
\vspace{2 mm}

	To each $g\in G$, we can associate a $C^*$-homomorphism  $\phi_g\in \S(G)$, where if $f\in \mathrm{RUC}_b(G)$, we set $\phi_g(f) = f(g)$. The map $g\to \phi_g$ is an embedding with dense image, and we typically identify $G$ with its image under this embedding.

	The group $G$ acts on $\mathrm{RUC}_b(G)$ on the right, where given $f\in \mathrm{RUC}_b(G)$ and $g, h\in G$, we set $(f\cdot g)(h) = f(gh)$. This gives rise to a continuous left-action on $\S(G)$, where given $g\in G$, $\phi\in \S(G)$, and $f\in \mathrm{RUC}_b(G)$, we set $(g\cdot \phi)(f) = \phi(f\cdot g)$. Since $\S(G)$ is compact, this gives it the structure of a $G$-flow.
	\vspace{2 mm}

	\begin{fact}
		\label{Fact:SGUniversal}
		If $X$ is a compact space and $f\colon G\to X$ is right uniformly continuous, then $f$ admits a continuous extension $\hatf$ to all of $\S(G)$. Conversely, if $f\colon \S(G)\to X$ is continuous, then $f|_G$ is right uniformly continuous.

		In particular, if $X$ is a $G$-flow and $x\in X$, then the map $\rho_x\colon G\to X$ continuously extends to a $G$-map $\hat{\rho}_x\colon \S(G)\to X$.

		It follows that any minimal subflow of $\S(G)$ is a universal minimal flow for $G$. This gives the existence part of \Cref{Fact:UMF}.
	\end{fact}
\vspace{2 mm}

When $X$ is a $G$-flow, $x\in X$, and $p\in \S(G)$, we often write $px$ instead of $\hat{\rho}_x(p)$. Note that $px = \lim_{g_i\to p} g_ix$. This shorthand ``multiplicative" notation becomes particularly suggestive when $X = \S(G)$.
\vspace{2 mm}

	\begin{fact}
		\label{Fact:SGSemigroup}
		On $\S(G)$, the binary operation given by $(p, q)\to pq\coloneqq \hat{\rho}_q(p)$ is associative. This turns $\S(G)$ into a \emph{compact right-topological semigroup}, which in this case means exactly that the right multiplication maps $\hat{\rho}_q$ are continuous for each $q\in \S(G)$. The following are basic facts about compact right-topological semigroups, most of which can be found in \cite{HS}.

		\begin{enumerate}
			\item
			For any $q\in \S(G)$, the right multiplication map $\hat{\rho}_q\colon \S(G)\to \S(G)$ is continuous.
			\item
			For any $g\in G$, the left multiplication map $\lambda_g\colon \S(G)\to \S(G)$, $p \mapsto gp$ is continuous.
			\item
			A \emph{left ideal} is any subset $L\subseteq \S(G)$ with $\S(G)\cdot L\subseteq L$. If $p\in \S(G)$, the left ideal $\S(G)\cdot p$ is closed. Minimal left ideals exist and are always closed. Minimal left ideals are exactly the minimal subflows of $\S(G)$.
			\item
			Every minimal left ideal $M$ contains an \emph{idempotent}, an element $u\in M$ with $uu = u$. Every other $p\in M$ satisfies $pu = p$. Hence the map $\hat{\rho}_u\colon \S(G)\to M$ is a retraction onto $M$.
			\item
			Every $G$-map $\phi\colon M\to N$ between minimal subflows of $\S(G)$ has the form $\hat{\rho}_q$ for some $q\in N$. Each such map is an isomorphism. Hence every minimal subflow of $\S(G)$ is isomorphic, showing the uniqueness part of \Cref{Fact:UMF}.
		\end{enumerate}
	\end{fact}
	\vspace{2 mm}

	We now discuss a more combinatorial construction of $\S(G)$, which is close in spirit to the original construction of the Samuel compactification of any uniform space \cite{Sam}. See also \cite{KoS} for the specific case of topological groups. We follow the presentation of \cite{ZucThesis}.
	\vspace{2 mm}

	\begin{defin}
		\label{Def:NearUlt}
		A collection $\cal{F}\subseteq \cal{P}(G)$ has the \emph{near finite intersection property}, or NFIP, if given any $k < \omega$, $A_0,..., A_{k-1}\in \cal{F}$ and any $U\in \cal{N}_{G}$, we have $\bigcap_{i< k} UA_i\neq \emptyset$.

		We say that $p\subseteq \cal{P}(G)$ is a \emph{near ultrafilter} if $p$ is maximal with respect to having the NFIP. Write $S_G$ for the collection of near ultrafilters on $G$.
	\end{defin}
\vspace{2 mm}

	The notation $S_G$ is temporary; see the fact below. We endow $S_G$ with the topology whose basic closed set has the form
	$$C_A \coloneqq \{p\in S_G: A\in p\}.$$
	The group $G$ acts on $S_G$ in the obvious fashion, where $A\in gp$ if and only if $g^{-1}A\in p$.
	\vspace{2 mm}

	\begin{fact}
		\label{Fact:NUltSG}
		$S_G\cong \S(G)$.
	\end{fact}
	\vspace{2 mm}

	Therefore we will retire the notation $S_G$, and simply think of $\S(G)$ as the space of near ultrafilters on $G$. Below we record some basic facts about near ultrafilters on $G$. Proofs can be found in \cite{ZucThesis}.
	\vspace{2 mm}

	\begin{fact}\mbox{}
		\label{Fact:NUltBasics}
		\vspace{-2 mm}

		\begin{enumerate}
			\item
			If $p\in \S(G)$ and $A\subseteq G$ with $A\not\in p$, then for some $U\in \cal{N}_G$ we have $UA\not\in p$. In particular, $A\in p$ if and only if $\overline{A}\in p$.
			\item
			If $p\in \S(G)$, then a basis of (not necessarily open) neighborhoods of $p$ is given by $\{C_{UA}: A\in p, U\in \cal{N}_G\}$.
			\item
			Suppose $X$ is compact and $f\colon G\to X$ is right uniformly continuous. Let $\hatf\colon \S(G)\to X$ be the continuous extension. Then for $p\in \S(G)$, we have $\hatf(p) = x$  if and only if for every $A\ni_{op} x$, we have $f^{-1}(A)\in p$.
			\item
			We view $G$ as a subset of $\S(G)$ by identifying $g\in G$ with the near ultrafilter $\{A\subseteq G: g\in \overline{A}\}$. Then given $A\subseteq G$ and letting $cl_{\S(G)}(A)$ denote the closure of $A$ in $\S(G)$, we have $cl_{\S(G)}(A) = C_A$. In particular,  if $\{A_i: i\in I\}$ are subsets of $G$ and $A = \bigcup_{i\in I} A_i$, then $\overline{\bigcup_{i\in I} C_{A_i}} = C_A$.
		\end{enumerate}
	\end{fact}

	\section{CAP groups}
	\label{Sec:CAP}
	In this section, we introduce the class of topological groups which we will investigate for the rest of the paper.
	These are defined in terms of the behavior of almost periodic points in $G$-flows.
	\vspace{2 mm}

	\begin{defin}
		\label{Def:AP}
		Let $G$ be a topological group and $X$ a $G$-flow. The set of \emph{almost periodic} points of $X$ is the set
		$$\ap(X) \coloneqq \{x\in X: \overline{Gx} \text{ is minimal}\}.$$
		If multiple groups act on $X$, we can write $\ap_G(X)$ to emphasize which group is being referred to. On $\ap(X)$, we let $E_G$ denote the equivalence relation of belonging to the same minimal subflow.
	\end{defin}
	\vspace{0 mm}

	\begin{defin}
		\label{Def:CAPGroup}
		A topological group $G$ has the \emph{closed AP} property, or is \emph{CAP}, if for any $G$-flow $X$, the set $\ap(X)\subseteq X$ is closed. In particular, $\ap(X)$ is a subflow of $X$. We say that $G$ is \emph{strongly CAP} if it is CAP and for any $G$-flow $X$, the equivalence relation $E_G\subseteq \ap(X)\times \ap(X)$ is closed.
	\end{defin}
	\vspace{2 mm}

In \Cref{Thm:EquivalentInG}, we will see that the notions of CAP and strongly CAP are equivalent. However, we do not have a ``direct" proof of this; instead, we will define the notion of a \emph{UEB group} and show that CAP groups are UEB and that UEB groups are strongly CAP.

  Recall that a topological group is \emph{pre-compact} if it is isomorphic to a dense subgroup of a compact group.
  \vspace{2 mm}

  \begin{exa}
		\label{Exa:PreCompactGroupsCAP}
		Every pre-compact group is CAP. To see why, suppose $G\subseteq K$ is dense, with $K$ a compact group. We note by \Cref{Fact:DenseSubgroup} that every $G$-flow is also a $K$-flow. As every $K$-orbit is closed, we see that $\ap_G(X) = \ap_K(X) = X$ for every $G$-flow $X$.
	\end{exa}
\vspace{-2 pt}

	\begin{exa}
		\label{Exa:LCGroupsNotCAP}
		Suppose that $G$ is locally compact and non-compact. Then $G$ is not CAP. To see this, consider the flow $2^G$ of closed subsets of $G$ with the Fell topology.
		Recall that a subbasis for the Fell topology is given by the sets $\{F \in 2^G : F \cap U \ne \emptyset, F \cap K = \emptyset\}$, for $U \subseteq G$ open and $K \subseteq G$ compact.
		 The group $G$ acts on $2^G$ by left multiplication. If $D\subseteq G$ is a pre-compact, symmetric open subset of $G$ containing the identity, we say that $S\subseteq G$ is $D$-spaced if for any $g\neq h\in S$, we have $gD\cap hD = \emptyset$. Let $Y_D\subseteq 2^G$ denote the subflow of $D$-spaced subsets of $G$. We note that $\ap(Y_D)\setminus \{\emptyset\} \neq \emptyset$; one can for instance fix $S\subseteq G$  a maximal $D$-spaced subset and show that $\emptyset\not\in \overline {G\cdot S}$.

		For each $D\subseteq G$ as above, find $S_D\in \ap(Y_D)\subseteq \ap(2^G)$ with $1_G\in S_D$. Viewing the collection of such $D$ as a directed set, let $S$ be a limit point of the $S_D$. Then we must have $S = \{1_G\}\not\in \ap(2^G)$; the last non-inclusion holds as $G$ is not compact.
	\end{exa}
\vspace{-2 pt}

	\begin{exa}
		\label{Exa:CAPiffMGMetr}
		When $G$ is metrizable, the notion coincides with having metrizable universal minimal flow. The reverse direction is due to \cite{JZ}, and the forward direction is due to Barto\v{s}ov\'a and Zucker and appears in \cite{ZucThesis}. However, the proof given there has a minor error, which we take the opportunity to fix in the proof of \Cref{Thm:EquivalentInG}.
	\end{exa}
\vspace{2 mm}

	The proof of many of our results on CAP groups rests on another characterization of CAP of a very different nature, which we introduce in \Cref{Sec:UEB}. The definition we give there, that of a \emph{UEB group}, is inspired by the result of \cite{BYMT} characterizing when $\M(G)$ is metrizable for a Polish group in terms of a canonical, but possibly not compatible, metric on $\M(G)$.

	\section{Topo-uniform spaces and the UEB uniformity}
	\label{Sec:TopoUnif}

	In this section, we discuss a uniformity one can put on $\S(G)$ called the \emph{UEB uniformity}. Though in general this uniformity is finer than the compact topology on $\S(G)$, it interacts with it in a strong way: together they form a \emph{topo-uniform space}. We will give several equivalent descriptions of this uniformity, and prove they are equivalent. This also gives us a uniformity on any minimal subflow of $\S(G)$;  we will see that this is in fact independent of the choice of minimal subflow, giving us a canonical uniformity on $\M(G)$. In the sections which follow, we will be particularly interested in those $G$ for which this uniformity on $\M(G)$ coincides with the compact topology; these are the groups which we will call \emph{UEB groups}, which we prove are exactly the CAP groups.

	\subsection{Topo-uniform spaces}
	\label{SubSec:TopoUnif}

	Topo-uniform spaces are a generalization of topo-metric spaces, which were introduced in \cite{BY} and have recently been employed to study the dynamics of Polish groups. Throughout this section, we will make frequent use of the notation defined in Subsection~\ref{SubSec:UnifSpaces}
	\vspace{2 mm}

	\begin{defin}
		\label{Defin:TopoUnif}
		A \emph{topo-uniform} space is a triple $(X, \tau, \mathcal{U})$, where $X$ is a set, $\tau$ is a topology on $X$ and $\mathcal{U}$ is a Hausdorff uniformity on $X$ such that:
		\begin{enumerate}
			\item
			\label{itm:lscUnif}
			$\mathcal{U}$ has a basis of $(\tau \times \tau)$-closed entourages;
			\item
			\label{itm:finerUnif}
			$O \in \mathcal{U}$ for any $(\tau\times\tau)$-open neighborhood $O$ of the diagonal in $X^2$.
		\end{enumerate}
		We refer to the symmetric $(\tau \times \tau)$-closed entourages of $\mathcal{U}$ as \emph{basic}.
	\end{defin}
	\vspace{0 mm}

	\begin{lemma}
		\label{Lem:IfHausdorffFiner}
		If $(X, \tau)$ is compact and $\mathcal{U}$ is a Hausdorff uniformity on $X$ satisfying item~\ref{itm:lscUnif} of \Cref{Defin:TopoUnif}, then $(X, \tau, \cal{U})$ is a topo-uniform space. Furthermore, $\cal{U}$ is complete. 
	\end{lemma}

	\begin{proof}
		Let $\Delta$ denote the diagonal of $X$ and fix a $(\tau\times\tau)$-open neighborhood $O\supseteq \Delta$. Since $\mathcal{U}$ is Hausdorff and satisfies item~\ref{itm:lscUnif}, we have $\Delta = \bigcap_{U \in \mathcal{U}}  \overline{U}$. By compactness of $\tau \times \tau$, there are $U_{1}, \dots, U_{n} \in \mathcal U$ such that $\overline{U_{1}} \cap \dots \cap \overline{U_{n}} \subseteq O$, so $O \in \mathcal{U}$.
		
		For the second claim, suppose that $(x_i)_{i\in I}$ is a $\cal U$-Cauchy net indexed by some directed set $I$. 
		For each basic $U \in \cal U$, fix $i_U\in I$ such that $(x_j, x_{j'}) \in U$ for all $j,j' \ge i_U$. 
		Then $x_{i_U}[U]$ is closed, thus compact, and the collection $\{x_{i_U}[U]: U\in \cal{U} \text{ basic}\}$ has the finite intersection property. 
		Therefore $\bigcap x_{i_U}[U] \ne \emptyset$. 
		Let $x$ belong to the intersection and fix $V \in \cal U$. 
		Let $U$ be a basic entourage such that $U^2 \subseteq V$.
		For all $j \ge i_U$, it holds that $(x_j, x_{i_U}) \in U$. 
		Since $x \in x_{i_U}[U]$, that is $(x_{i_{U}}, x) \in U$, it follows that $(x_j, x) \in U^2 \subseteq V$, from which we conclude that $x_i\xrightarrow{\cal{U}} x$. 
	\end{proof}
	\vspace{2 mm}

	\noindent
	From now on, we assume that $(X, \tau)$ is compact.

	\begin{rem}
		\label{Rem:xUisClosed}
		For each $x \in X$ and each basic $U \in \mathcal{U}$, it holds that $x[U]$ is $\tau$-closed, as it is the projection of the closed set $(\set{x}\times X)\cap U$. Moreover, for each $A \ni_{op} x$, there is a basic $U \in \mathcal{U}$ such that $x[U] \subseteq A$ by item \ref{itm:finerUnif} of \Cref{Defin:TopoUnif}.
	\end{rem}
	\vspace{2 mm}

	\begin{defin}
		\label{Defin:TopoUnifAdequate}
		A compact topo-uniform $(X, \tau, \mathcal{U})$ space is \emph{adequate} if for a base of $U \in \mathcal{U}$ and each $\tau$-open $A \subseteq X$, it holds that  $A(U)$ is $\tau$-open.
	\end{defin}
	\vspace{0 mm}

	\begin{lemma}
		\label{Lem:ANWDBall}
		If $(X, \tau, \mathcal{U})$ is a an adequate topo-uniform space such that the topology induced by $\mathcal{U}$ is strictly finer that $\tau$, then there exists $x \in X$ and a basic $U \in \mathcal{U}$ such that $x[U]$ is $\tau$-nowhere dense.
	\end{lemma}

	\begin{proof}
		Since each such $x[U]$ is closed, it is enough to show that there is one with empty interior. Suppose towards a contradiction that for each $x\in X$ and  $U\in \mathcal{U}$, we had that $\interior(x[U]) \neq \emptyset$. We show that for each $x\in X$, $\setnew{x[U]}{ U \in \mathcal{U}}$ is a $\tau$-neighborhood basis at $x$, thus contradicting the assumption that the uniform topology is finer than $\tau$. So fix $x\in X$ and $B \ni_{op} x$. Let $U \in \mathcal{U}$ be such that $x[U] \subseteq B$. Let $V\in \mathcal{U}$ be basic and symmetric and such that $V^4\subseteq U$. Call $A = \interior(x[V])$ and consider $A(V^3)$, which is open by adequacy. For each $y \in A$, we have $(y, x) \in V \ll V^3$, so $x \in A(V^3)$. On the other hand, fix $z \in A(V^3)$. Then $z\in A[W]$ for some $W\ll V^3$. But $A \subseteq x[V]$, so $(x, z) \in VW \subseteq V^4 \subseteq U$, so $z \in x[U]$. Therefore $x \in A(V^3) \subseteq x[U]$, which shows that $x[U]$ is a neighborhood of $x$.
	\end{proof}

	\subsection{The UEB uniformity}
	\label{Subsec:UEB}
	We now proceed to give three descriptions of the UEB uniformity and prove that these descriptions are equivalent.

	\subsubsection{Entourages via UEB sets}
	The first definition of the UEB uniformity we give is the classical one; see \cite{Pachl} or \cite{ST} for a more detailed exposition.
	\vspace{2 mm}

	\begin{defin}\mbox{}
		\label{Def:UEBTopology}
		\vspace{-3 mm}

		\begin{enumerate}
			\item
			A subset $H\subseteq \mathrm{RUC}_b(G)$ is a \emph{uniformly equicontinuous and bounded} set, or \emph{UEB} set, if both $\sup\{\|f\|: f\in H\}< \infty$ (uniformly bounded) and for every $\epsilon > 0$ there is $U\ni_{op} G$ so that for any $g, h\in G$ with $gh^{-1}\in U$, we have $|f(g)-f(h)|<\epsilon$ (uniformly equicontinuous).
			\item
			The \emph{UEB uniformity} on $\S(G)$ is generated by entourages of the form
			$$[H, \epsilon] \coloneqq \{(p, q)\in \S(G)\times \S(G): |p(f)-q(f)|<\epsilon \text{ for all } f\in H\}$$
			where $H \subseteq \mathrm{RUC}_{b}(G)$ is a UEB set and $\epsilon>0$.
		\end{enumerate}
	\end{defin}
	We call the topology induced by the UEB uniformity the \emph{UEB topology}.

	\subsubsection{Entourages via identity neighborhoods}

	The second definition is inspired by \cite{ZucMHP} and can be defined on a wide class of $G$-flows.
	\vspace{2 mm}

	\begin{defin}
		\label{Def:MHP}
		A $G$-flow $X$ is called \emph{maximally highly proximal}, or MHP, if whenever $A \subseteq_{op}X$ and $x\in \overline{A}$, then $x \in \interior(\overline{UA})$ for each $U\in \cal{N}_G$.
	\end{defin}
	\vspace{2 mm}

	In this work, we will be mostly interested in two particular MHP flows, namely $\S(G)$ and $\M(G)$.

	\begin{fact}[\cite{ZucMHP}]
		\label{Fact:MGandSGareMPH}
	For any topological group $G$, the Samuel compactification $\S(G)$ and the universal minimal flow $\M(G)$ are MHP $G$-flows. 
	\end{fact}

	Given $U, V\in \cal{N}_G$, we write $V\ll U$  if and only if there is $W\in \cal{N}_G$ with $WVW\subseteq U$. This is reminiscent of the definition given earlier for entourages of a uniform space.
	\vspace{2 mm}

	\begin{defin}
		\label{Defin:BasicEntourageOnMHP}
		Let $X$ be an MHP $G$-flow. For each $U \in\cal{N}_G$, we define $\onx{U} \subseteq X^{2}$ by letting $(x, y) \in \onx{U}$ if and only if for all $A \ni_{op} x$ and all $U_0\in \cal{N}_G$ with $U\ll U_0$, we have $y \in \overline{U_0A}$. Let $\mathcal{U}_{X}$ denote the the filter of subsets of $X^{2}$ generated by $\setnew{\onx{U}}{U\in \cal{N}_G}$.
	\end{defin}
	\vspace{2 mm}

\noindent
	In Definition~\ref{Defin:BasicEntourageOnMHP}, notice that $y\in \overline{U_0A}$ iff for every $B\ni_{op} y$, we have $U_0A\cap B\neq \emptyset$. Since $U_0$ is symmetric, we see that $\onx{U}$ is symmetric. Before proving that $\mathcal{U}_{X}$ is indeed a uniformity, we need the following property of the entourages $\onx{U}$.
	\vspace{2 mm}

	\begin{lemma}
		\label{Lem:EquivDefBasicEntourageOnMHP}
		Let $X$ be an MHP $G$-flow, and fix $U \in \cal{N}_G$. Then $(x, y) \in \onx{U}$ if and only if for all $A \subseteq_{op} X$ with $x \in \overline{A}$ and all $U_0\in \cal{N}_G$ with $U\ll U_0$, we have $y \in \interior(\overline{U_0A})$.
	\end{lemma}

	\begin{proof}
		One direction is clear. For the other, suppose that $x \in \overline{A}$ for some $A \subseteq_{op} X$, and fix $U_0\in \cal{N}_G$ with $U\ll U_0$. Find $U_1\in \cal{N}_G$ with $U\ll U_1\ll U_0$, and let $W\in \cal{N}_G$ be such that $WU_1W \subseteq U_0$. Since $X$ is MHP, it follows that $x \in \interior(\overline{WA})$. By the definition of $\onx{U}$, it holds that $y \in \overline{U_1\interior(\overline{WA})} \subseteq \overline{U_1\overline{WA}} \subseteq \overline{U_1WA}$. By MHP again, $y \in \interior(\overline{WU_1WA}) \subseteq \interior(\overline{U_0A})$.
	\end{proof}
	\vspace{2 mm}

\noindent
	\Cref{Lem:EquivDefBasicEntourageOnMHP} implies that if $U\in \mathcal{N}_G$, then for any $V\in \cal{N}_G$, we have $\onx{U}\onx{U}\subseteq \onx{UVU}$. This together with our earlier observation that each $\onx{U}$ is symmetric yields the following.
	\vspace{2 mm}
	
	\begin{prop}
		\label{Prop:UniformityOnMHP}
		$\cal{U}_X$ is a uniformity on $X$.
	\end{prop}
	\vspace{2 mm}
	
	\noindent
	Now \Cref{Lem:EquivDefBasicEntourageOnMHP} implies that whenever $U, V\in \cal{N}_G$ with $V\ll U$, we also have $\onx{V}\ll\onx{U}$.
	\vspace{2 mm}

	\begin{theorem}
		\label{Thm:XUisTopoUnif}
		Let $(X, \tau)$ be an MHP $G$-flow. Then  $(X, \tau, \mathcal{U}_{X})$ is an adequate topo-uniform space.
	\end{theorem}

	\begin{proof}
		The theorem statement splits into two claims.
		\vspace{2 mm}

		\begin{claim}
			\label{Claim:UisTopoUnif}
			$(X, \tau, \mathcal{U}_{X})$ is a topo-uniform space.
		\end{claim}

		\begin{proof}[Proof of Claim]
			By \Cref{Lem:IfHausdorffFiner}, it is enough to show that $\cal{U}_X$ is Hausdorff and admits a base of $(\tau\times\tau)$-closed entourages. First we show that for each $U\in \cal{N}_G$, $\onx{U}$ is $(\tau \times \tau)$-closed. Let $x_{i} \to x$, $y_{i} \to y$ be nets in $X$ such that for each $i$, $(x_{i}, y_{i}) \in \onx{U}$. Let $A \ni_{op} x$ and $U_0\in \cal{N}_G$ with $U\ll U_0$; then eventually $x_{i} \in A$, so $y_{i} \in \overline{U_0A}$. But this is a $\tau$-closed set, so $y \in \overline{U_0A}$, and thus $(x, y) \in \onx{U}$.

			Now we show that $\mathcal{U}_{X}$ is Hausdorff. Let $x, y \in X$ and let $A \ni_{op} x$ be such that  $y \not \in \overline{A}$. Find $U \in \cal{N}_G$ be such that $y \not \in \overline{U^{3}A}$. Then $(x, y) \not \in \onx{U}$.
		\end{proof}
	\vspace{0 mm}

		\begin{claim}
			\label{Claim:UisAdequate}
			$(X, \tau, \mathcal{U}_{X})$ is adequate.
		\end{claim}

		\begin{proof}[Proof of Claim]
			We will show for every $A\subseteq_{op} X$ and $U\in \cal{N}_G$ that $A(\onx{U})$ is open. We do this by mimicking the proof of Theorem 4.8 from \cite{ZucMHP}.

			We first note that for any sets $A, A_i\subseteq X$ with $A = \bigcup_i A_i$ that $A(\onx{U}) = \bigcup_i A_i(\onx{U})$. Now fix $A\subseteq_{op} X$. As $(X, \tau)$ is compact Hausdorff, we can write $A = \bigcup_i A_i$ with each $A_i$ a \emph{regular open} set, i.e.\ with $A_i = \mathrm{int}(\overline{A_i})$. So it suffices to prove the claim when $A\subseteq_{op} X$ is a regular open set. Write $K = X\setminus A$. Given $U\in \cal{N}_G$, we show that:
			\[
			A(\onx{U}) = \bigcup_{\begin{subarray}{c} V \ll U\\ W\in \cal{N}_G \end{subarray}} \interior\left(\overline{V\cdot(X\setminus \overline{WK})}\right).
			\]
			Suppose that $x \in A(\onx{U})$, and find $V_0 \ll U$, and $y \in A$, such that $(x, y) \in \onx{V_0}$. Find $V_1\in \cal{N}_G$ with $V_0\ll V_1\ll U$. Since $y\in A$, we can find $W\in \cal{N}_G$ with $y\in X\setminus \overline{WK}$. Then since $(x, y)\in \onx{V_0}$ and $V_0\ll V_1$, we have by \Cref{Lem:EquivDefBasicEntourageOnMHP} that $x\in \mathrm{int}\left(\overline{V_1\cdot (X\setminus \overline{WK})}\right)$.

			Conversely, suppose for some $V\ll U$ and $W\in \cal{N}_G$ that $x\in \mathrm{int}\left(\overline{V\cdot (X\setminus \overline{WK})}\right)$. It follows that for any $B\ni_{op} x$ that $VB\cap \left(X\setminus \overline{WK}\right) \neq \emptyset$. Therefore we have:
			$$\left(\overline{X\setminus \overline{WK}}\right)\cap \bigcap \{VB: B\ni_{op} x\} \neq \emptyset.$$
			Pick $y$ from this set. Then $(x, y)\in \onx{V}$. Towards a contradiction, suppose $y\not\in A$, i.e.\ that $y\in K$. As $A$ is regular open, we have that $y\in \mathrm{int}(\overline{WK})$. But since $y\in \overline{X\setminus \overline{WK}}$, this is a contradiction.
		\end{proof}
		These two claims conclude the proof of \Cref{Thm:XUisTopoUnif}.
	\end{proof}
	\vspace{2 mm}

	Notice that we have two ways of defining this uniformity on $\M(G)$: either directly via \Cref{Defin:BasicEntourageOnMHP}, since $\M(G)$ is an MHP flow, or by restricting $\mathcal{U}_{\S(G)}$ to a minimal subflow.
	These are in fact the same:

	\begin{prop}
		\label{Prop:UnifMG}
		Suppose $M\subseteq \S(G)$ is a minimal subflow, and let $p, q\in M$ and $U\in \cal{N}_G$.
		The following are equivalent.
		\begin{enumerate}
			\item
			$(p, q) \in \onx{U}$ computed in $\S(G)$
			\item
			$(p, q) \in \onx{U}$ computed in $M$.
		\end{enumerate}
	\end{prop}

	\begin{proof}
		First suppose item $2$ holds.
		If $B\ni p$ is an open subset of $\S(G)$, then $B\cap M$ is relatively open in $M$, so for any $U_0 \in \cal{N}_{G}$ with $U\ll U_0$, we have $q\in \overline{U_0(B\cap M)}\subseteq \overline{U_0 B}$, so item $1$ holds.

		For the other direction, suppose item $1$ holds. Fixing an idempotent $u\in M$, we obtain a continuous, $G$-equivariant retraction $\hat{\rho}_u\colon \S(G)\to M$. If $B\subseteq M$ is a neighborhood of $p$ in $M$, then $\hat{\rho}_u^{-1}(B)$ is a neighborhood of $p$ in $\S(G)$.
		So we have for any $U_0\in \cal{N}_G$ with $U\ll U_0$ that $q\in \overline{U_0 \hat{\rho}_u^{-1}(B)} \subseteq \hat{\rho}_u^{-1}\left(\overline{U_0B}\right)$.
		Because $\hat{\rho}_u$ is a retraction, this implies that $q\in \overline{U_0 B}$ as desired.
	\end{proof}
	\vspace{0 mm}

	\begin{cor}
		\label{Cor:UnifMGSame}
		Whenever $M, N\subseteq \S(G)$ are minimal $G$-flows and $\phi\colon M\to N$ is a $G$-map, then in fact $\phi$ is a  $\mathcal{U}_{\M(G)}$-uniform isomorphism.
	\end{cor}

	\begin{proof}
		This follows from the fact that $\phi$ is a $G$-isomorphism (\Cref{Fact:SGSemigroup}).
	\end{proof}

	\subsubsection{Entourages via pseudo-metrics}

	The third definition is inspired by \cite{BYMT}. There, the authors start with a Polish group $G$ equipped with a compatible, bounded, right-invariant metric and endow $\S(G)$ with a topo-metric structure by considering the Lipschitz functions on $G$. Rather than directly creating a topo-uniform structure on $\S(G)$ as in the previous definition, we can instead create a family of ``topo-pseudo-metrics."
	\vspace{2 mm}

	\begin{defin}
		\label{Def:LipschitzPseudo}
		Suppose $d$ is a continuous, diameter $1$, right-invariant pseudo-metric on $G$. We define the pseudo-metric $\partial_d$ on $\S(G)$ by setting
		$$\partial_d(p, q) \coloneqq \sup\{|f(p)-f(q)|: f\in \lip(d)\}.$$
	\end{defin}
	\vspace{2 mm}

\noindent
	We can nicely characterize $\partial_d$ by viewing $\S(G)$ as the space of near ultrafilters on $G$. The next proposition is adapted from \cite{ZucMHP}.
	\vspace{2 mm}

	\begin{prop}
		\label{Prop:MHPPseudo}
		Suppose $p, q\in \S(G)$ and $c\geq 0$. Then the following are equivalent:
		\begin{enumerate}
			\item
			$\partial_d(p, q) \leq c$
			\item
			For any open $B\subseteq \S(G)$ with $p\in B$ and any $\epsilon > 0$, we have $q \in \overline{d(c+\epsilon)\cdot B}$.
			\item
			For any $A\subseteq G$ with $A\in p$ and any $\epsilon > 0$, we have $d(c+\epsilon)\cdot A\in q$.
		\end{enumerate}
	\end{prop}

	\begin{proof}
		As a preliminary remark, note that if $f\in \lip(d)$ is continuously extended to $\hatf\colon \S(G)\to [0,1]$, then the continuous extension is \emph{orbit $d$-Lipschitz}, i.e.\ for any $g\in G$ and $p\in \S(G)$, we have $|\hatf(gp)-\hatf(p)|\leq d(g, 1_G)$.

		$(\neg 1\Rightarrow \neg 2)$ Suppose that item $1$ fails, witnessed by some $f\in \lip(d)$. Letting $\hatf$ denote the continuous extension, suppose for some $\epsilon > 0$ that $|\hatf(p)-\hatf(q)|> c+2\epsilon$. Find $B\ni_{op} p$ so that $|\hatf(p)-\hatf(p_0)|< \epsilon$ whenever $p_0\in B$. Then for $p_0\in B$ and $g\in d(c+\epsilon)$, we have $|\hatf(p) - \hatf(gp_0)|< c+2\epsilon$. So if $p_1\in \overline{d(c+\epsilon)\cdot B}$, we have $|\hatf(p) - \hatf(p_1)|\leq c+2\epsilon$. It follows that $q\not\in \overline{d(c+\epsilon)\cdot B}$.

		$(\neg 2\Rightarrow \neg 3)$ Suppose that item $2$ fails, witnessed by some neighborhood $B$ of $p$ and some $\epsilon > 0$. Set $A = B\cap G$. Then $\overline{B} = C_A$, and by \Cref{Fact:NUltBasics}, we have $\overline{d(c+\epsilon)\cdot B} = \overline{d(c+\epsilon)\cdot C_A} = C_{d(c+\epsilon)\cdot A}$. So $d(c+\epsilon)\cdot A\not\in q$.

		$(\neg 3\Rightarrow \neg 1)$ Suppose that item $3$ fails, witnessed by some $A\in p$ and $\epsilon > 0$. Define $f\colon G\to [0,1]$ via $f(g) = d(g, A)$. Then $f\in \lip(d)$, and upon continuously extending to $\S(G)$, we have $\hatf(p) = 0$. Towards a contradiction, suppose that $\hatf(q)\leq c$. Then by \Cref{Fact:NUltBasics} we would have $f^{-1}([0, c+\epsilon])\in q$. However, $f^{-1}([0, c+\epsilon]) = d(c+\epsilon)\cdot A\not\in q$.
	\end{proof}
	\vspace{0 mm}

	\begin{cor}
		\label{Cor:OrbitLipschitz}
		For any $p\in \S(G)$ and $g, h\in G$, we have $\partial_d(gp, hp)\leq d(g,h)$.
	\end{cor}

	\begin{proof}
		We note that $hp = (hg^{-1})gp \in d( d(g, h) + \epsilon) gp$ for each $\epsilon >0$ and apply \Cref{Prop:MHPPseudo}.
	\end{proof}
	\vspace{2 mm}

	Item $2$ of \Cref{Prop:MHPPseudo} in fact defines a pseudo-metric on any MHP flow (one needs to allow $\partial_d$ to take the value $+\infty$ for MHP flows which are not topologically transitive). In particular, $\partial_d$ can be defined on $\M(G)$ in this way. Much as in \Cref{Prop:UnifMG} and \Cref{Cor:UnifMGSame}, we have that whenever $M\subseteq \S(G)$ is a minimal subflow and $p, q\in M$, the value of $\partial_d(p, q)$ is the same computed in $\S(G)$ or computed in $M$, and any $G$-map between two minimal subflows of $\S(G)$ is a $\partial_d$-isometry.

	Recall that if $\tau$ is a compact topology on $X$ and $\partial$ is a pseudo-metric on $X$, we say that $\partial$ is \emph{$\tau$-lower-semi-continuous}, or $\tau$-lsc, if for every $c\geq 0$, we have $\{(x, y): \partial(x, y)\leq c\}\subseteq X\times X$ is $(\tau\times\tau)$-closed. This is the analog for pseudo-metrics of item $1$ of \Cref{Defin:TopoUnif}.

	For now, we let $\tau$ denote the compact topology on $\S(G)$.
	As a rule of thumb, topological vocabulary will refer to $\tau$ unless specifically indicated otherwise.
	\vspace{2 mm}

	\begin{prop}\mbox{}
		\label{Prop:TauLSC}
		\vspace{-2 mm}

		\begin{enumerate}
			\item
			Each of the pseudo-metrics $\partial_d$ is $\tau$-lsc.
			\item
			If $p\neq q\in \S(G)$, then there is some continuous, diameter 1, right invariant pseudo-metric $d$ on $G$ with $\partial_d(p, q)> 0$.
		\end{enumerate}
	\end{prop}

	\begin{proof}
		For $1$, suppose $(p_i)_{i\in I}$ and $(q_i)_{i\in I}$ are nets from $\S(G)$ with $p_i\to p$ and $q_i\to q$, and suppose $\partial_d(p_i, q_i)\leq c$.
		If $f\in \lip(d)$, then $\hatf(p_i)\to \hatf(p)$ and $\hatf(q_i)\to \hatf(q)$, so $|\hatf(p)-\hatf(q)|\leq c$, and hence $\partial_d(p, q)\leq c$.

		For $2$, suppose $A\subseteq G$ with $A\in p$ and $A\not\in q$.
		So for some $U\in \cal{N}_G$ we have $UA\not\in q$.
		Find a pseudo-metric $d$ on $G$ with $d(1)\subseteq U$.
		By \Cref{Prop:MHPPseudo}, it follows that $\partial_d(p, q)\geq 1$.
	\end{proof}
	\vspace{2 mm}

	Recall that if $X$ is a compact space and $\partial$ is a pseudo-metric on $X$, we say that $\partial$ is \emph{adequate} if for any $\epsilon > 0$ and any open $A\subseteq X$, we have that $B_{\partial}(A, \epsilon) \coloneqq \{x\in X: \partial(x, A)< \epsilon\}$ is open.
	\vspace{2 mm}

	\begin{prop}
		\label{Prop:Adequacy}
		On $\S(G)$, each of the pseudo-metrics $\partial_d$ is adequate.
	\end{prop}

	\begin{proof}
		See Theorem 4.8 of \cite{ZucMHP}, recalling that $\S(G)$ is an MHP $G$-flow. 
	\end{proof}

	\subsubsection{Equivalent uniformities}

	%%%%

	The following proposition links the entourages defined in \Cref{Defin:BasicEntourageOnMHP} and those given by the pseudo-metrics on $\S(G)$.
	\vspace{2 mm}

	\begin{prop}
		\label{Prop:PseudoAndUnif}
		Let $d$ be a continuous, diameter $1$, right-invariant pseudo-metric on $G$ and $c>0$. Then $ \setnew{(p, q)}{\partial_{d}(p, q) < c} \subseteq \onx{d(c)} \subseteq \setnew{(p, q)}{\partial_{d}(p, q) \le c}$. In particular, $\partial_{d}$ is $\mathcal{U}_{\S(G)}$-uniformly continuous.
	\end{prop}

	\begin{proof}
		Suppose $\partial_{d}(p, q) < c$, and let $A \ni_{op} p$. Then $q \in \overline{d(c)A}$, and it follows that $(p, q) \in \onx{d(c)}$.

		For the other inclusion, suppose $(p, q) \in \onx{d(c)}$ and let $A \ni_{op} p $ and $\epsilon>0$. Noting that $d(c)\ll d(c+\epsilon)$, we have $q \in \overline{d(c +\epsilon)A}$, so $\partial_{d}(p, q) \le c$.
	\end{proof}
	\vspace{2 mm}

	%%%%

	As $d$ ranges over all continuous, diameter 1, right-invariant pseudo-metrics on $G$, we obtain a uniformity $\cal{U}_{met}$ on $\S(G)$ which is generated by the induced pseudo-metrics $\partial_d$.
	\vspace{2 mm}

	\begin{prop}
		\label{Prop:PseudoUEB}
		The uniformities $\cal{U}_{met}$, $\mathcal{U}_{\S(G)}$, and the UEB uniformity coincide.
	\end{prop}

	\begin{proof}
		For each continuous, diameter 1, right-invariant pseudo-metric $d$ on $G$, the space $\lip(d)\subseteq \mathrm{RUC}_b(G)$ is a UEB set, so $\setnew{(p, q)}{\partial_{d}(p, q) < c} = [\lip(d), c]$. Therefore $\cal{U}_{met}$ entourages are UEB entourages.

		Now, suppose $H\subseteq \mathrm{RUC}_b(G)$ is a UEB set and let $\epsilon>0$, with the objective of finding $U \in \cal{N}_G$ such that $\onx{U} \subseteq [H, \epsilon]$. Let $\phi_H\colon (0,1)\to \cal{N}_G$ describe the modulus of uniform equicontinuity. We remark that if $f\in H$ and $\hatf$ denotes the continuous extension to $\S(G)$, then if $g\in \phi_H(\delta)$ and $p\in \S(G)$, we have $|\hatf(p) - \hatf(gp)|\leq \delta$.

		Fix $U \in \cal{N}_G$ with $U \ll \phi_H(\epsilon/4)$. Suppose $p, q \in \S(G)$ are such that $(p,q) \in \onx{U}$. For each $f \in H$ let $B_{f} \ni_{op} p$ be such that $\lvert \hatf(p) -\hatf(p') \rvert < \epsilon/4$ for all $p' \in B_{f}$. For all $f \in H$ and each $U_0\in \cal{N}_G$ with $U\ll U_0$ it holds that $q \in \overline{U_0B_{f}}$, so in particular $q \in \overline{\phi_{H}(\epsilon/4)B_{f}}$. If $h_{i}q_{i} \in \phi_{H}(\epsilon/4)B_{f}$ is a net converging to $q$, we have $\lvert \hatf(p) -\hatf(h_{i}q_{i}) \rvert \le \lvert \hatf(p) - \hatf(q_{i}) \rvert + \lvert \hatf(q_{i}) - \hatf(h_{i}q_i) \rvert < \epsilon/2$. Therefore $\lvert \hatf(p) -\hatf(q) \rvert \le \epsilon/2$ and $(p, q) \in [H, \epsilon]$.
		We thus have that UEB entourages are $\cal{U}_{\S(G)}$ entourages. 

		The last direction, $\cal{U}_{\S(G)} \subseteq \cal{U}_{met}$, is taken care by \Cref{Prop:PseudoAndUnif}.
	\end{proof}

	\section{UEB groups}
	\label{Sec:UEB}
	In this section we define the class of UEB groups, based on the behavior of the UEB uniformity on $\M(G)$.
	In Section~\ref{Sec:EquivCAPUEB}, we will prove that a group is UEB  if and only if it is CAP. Before doing so, we collect and discuss results which make direct use of the UEB definition. We use the term ``UEB group" rather than ``CAP group" throughout this section to emphasize the methods involved in the proofs.
	\vspace{2 mm}

	\begin{defin}
		\label{Def:UEBGroup}
		A topological group $G$ is \emph{UEB} if the compact and UEB uniformities coincide on $\M(G)$.
	\end{defin}
	\vspace{-2 pt}

	\begin{exa}
		\label{Ex:PolishUEBGroups}
		If $G$ is a metrizable group, then the UEB uniformity on $\S(G)$ can be given by a single metric, namely $\partial_d$ for $d$ a compatible, right-invariant metric on $G$. In \cite{BYMT}, it is shown that if $\M(G)$ is metrizable, then this metric is a compatible metric for $\M(G)$. In other words, $G$ is UEB exactly when $\M(G)$ is metrizable.
	\end{exa}
	\vspace{2 mm}

	The work we did in Subsection~\ref{Subsec:UEB} will give us several ways to understand UEB groups. For instance, when working with the induced pseudo-metrics, we have the following simple proposition.
	\vspace{2 mm}

	\begin{prop}
		\label{Prop:CAPSimpleEquiv}
		For a topological group $G$, the following are equivalent.
		\begin{enumerate}
			\item
			$G$ is UEB.
			\item
			For any continuous, diameter $1$, right-invariant pseudo-metric $d$ on $G$, the pseudo-metric $\partial_d$ on $\M(G)$ is continuous.
			\item
			For some base $\cal{D}$ of continuous, diameter $1$, right-invariant pseudo-metrics on $G$ and any $d\in \cal{D}$, the pseudo-metric $\partial_d$ on $\M(G)$ is continuous.
		\end{enumerate}
	\end{prop}

	\begin{proof}
		Clearly $1\Rightarrow 2\Rightarrow 3$. For $3\Rightarrow 1$, suppose we are given $\cal{D}$ as in item $3$. We first remark that it is sufficient to show that the UEB and compact \emph{topologies} coincide, as there is a unique compatible uniform structure on a compact Hausdorff space. So let $x\in \M(G)$, and suppose $\Theta\in \cal{U}_{\M(G)}$. In particular, we can find some $U\in \cal{N}_G$ with $x[\onx{U}]\subseteq x[\Theta]$. Since $\cal{D}$ is a base of pseudo-metrics on $G$, we can find $d\in \cal{D}$ and $c> 0$ with $d(c)\subseteq U$. Then by \Cref{Prop:PseudoAndUnif} we have $\{y: \partial_d(x, y)< c\} \subseteq x[\onx{d(c)}]\subseteq x[\onx{U}]\subseteq x[\Theta]$ as desired.
	\end{proof}
	\vspace{0 mm}

	\begin{cor}
		\label{Cor:AutGroupCAP}
		Suppose $G$ is UEB. Then $\aut(\M(G))$ is a compact group.
	\end{cor}

	\begin{proof}
		By the remark after \Cref{Cor:OrbitLipschitz}, every automorphism of $\M(G)$ is a $\partial_d$-isometry for each of the pseudometrics $\partial_d$. Since each $\partial_d$ is continuous, the space $\M(G)/\partial_d$ is a compact metric space, and any element of $\aut(\M(G))$ induces an isometry of $\M(G)/\partial_d$. Therefore, we can view $\aut(\M(G))$ as a subgroup of the compact group $\prod_d \mathrm{Iso}(\M(G)/\partial_d)$.

		To see that $\aut(\M(G))$ forms a closed subgroup, suppose $\alpha_i\in \aut(\M(G))$, and suppose $\alpha_i\to \alpha\in \prod_d \mathrm{Iso}(\M(G)/\partial)$. Considering the homeomorhpism $\M(G)\to \prod_d \M(G)/\partial_d$, we see that $\alpha$ is the homeomorphism of $\M(G)$ given by $\alpha(x) = \lim_i \alpha_i(x)$. It follows that $\alpha_i$ commutes with the $G$-action since each $\alpha_i$ does.
	\end{proof}
	\vspace{2 mm}

	By considering the entourages $\onx{U}$ for $U\in \cal{N}_G$, we obtain a bound on the complexity of $\M(G)$ as a topological space when $G$ is UEB. Recall that a poset $\langle \bb{P}, \leq\rangle$ is \emph{directed} if for any $p, q\in \bb{P}$, there is some $r\in \bb{P}$ with $p, q\leq r$. A subset $S\subseteq \bb{P}$ is \emph{cofinal} if for every $p\in \bb{P}$, there is $q\in S$ with $p\leq q$. If $\bb{P}$ and $\bb{Q}$ are directed posets, then a map $f\colon \bb{P}\to \bb{Q}$ is \emph{cofinal} if the image of every cofinal subset of $\bb{P}$ is cofinal in $\bb{Q}$. The map $f$ is \emph{monotone} if it respects the poset orders. The directed posets we will consider are $\cal{N}_G$, ordered by reverse inclusion, and $\mathrm{Nbd}(\Delta_{\M(G)})$, the poset of neighborhoods of the diagonal $\Delta_{\M(G)}\subseteq \M(G)\times \M(G)$. By considering the map $U\to \onx{U}$ for $U\in \cal{N}_G$, we obtain:
	\vspace{2 mm}

	\begin{prop}
		\label{Prop:CAPTukey}
		Suppose $G$ is UEB. Then there is a monotone cofinal map from $\cal{N}_G$ to $\mathrm{Nbd}(\Delta_{\M(G)})$.
	\end{prop}
	\vspace{2 mm}

	Monotone, cofinal maps are a particularly nice form of \emph{Tukey reduction} (where a Tukey reduction would not require that the map be monotone). We remark that when $G$ is metrizable, \Cref{Prop:CAPTukey} says exactly that $\M(G)$ is metrizable, by \cite{Snei}.
	More generally, $\M(G)$ is metrizable for any UEB group $G$ such that $\cal N_G$ has a cofinal subset of type $\omega^{\omega}$, by \cite{CO}.

	On the other hand, when $G$ is not UEB, one can show that $\M(G)$ is a somewhat complicated topological space. The key ingredient is the following characterization of UEB groups.
	\vspace{2 mm}

	\begin{prop}
		\label{Prop:BadOpensNotCAP}
		Suppose $G$ is a topological group. Then the following are equivalent.
		\begin{enumerate}
			\item
			$G$ is UEB.
			\item
			For any sequence $\{A_n: n< \omega\}$ of non-empty open subsets of $\M(G)$ and any $U\in \cal{N}_G$, we have that $\{UA_n: n< \omega\}$ is not pairwise disjoint.
		\end{enumerate}
	\end{prop}

	\begin{proof}
		We proof the contrapositives. First assume that $G$ is not UEB. Find a continuous, diameter $1$, right-invariant pseudo-metric $d$ on $G$ so that $\partial_d\coloneqq \partial$ is not continuous. In particular, $(\M(G)/\partial, \partial)$ is strictly finer than the compact topology on $\M(G)/\partial$, so is not compact. Moving back up to $\M(G)$, we can find an infinite set $Y\subseteq \M(G)$ so that for some $c> 0$, we have $\partial(x, y)> 2c$ for every $x\neq y\in Y$. From here, we can find a sequence of $x_n\in Y$ and of $A_n\ni_{op} x_n$ so that $d(c/2)\cdot A_m\cap d(c/2)\cdot A_n\neq \emptyset$ whenever $m\neq n< \omega$. For more details, see Theorem 5.3 in \cite{ZucMHP}.

		For the other direction, suppose that there were non-empty open sets $A_n\subseteq \M(G)$ and $U\in \cal{N}_G$ with $\{UA_n: n< \omega\}$ pairwise disjoint. Find a continuous, diameter $1$, right-invariant pseudo-metric $d$ on $G$ with $d(1)\subseteq U$. Then if we pick a point $x_n\in A_n$ for each $n$, \Cref{Prop:MHPPseudo} tells us that $\partial(x_m, x_n)\geq 1$ whenever $m\neq n< \omega$. In particular, $\partial$ cannot be a continuous pseudo-metric on the compact space $\M(G)$.
	\end{proof}
	\vspace{0 mm}

	\begin{cor}
		\label{Cor:NotCAPEmbedBeta}
		If $G$ is not UEB, then $\M(G)$ embeds a copy of $\beta \omega$.
		In particular, a group $G$ with metrizable $\M(G)$ is CAP, regardless of metrizability of $G$.
	\end{cor}

	\begin{proof}
		Fix a sequence $\{A_n: n<\omega\}$ of non-empty open subsets of $\M(G)$ and $U\in \cal{N}_G$ so that the sequence $\{UA_n: n< \omega\}$ is pairwise disjoint. Pick $x_n\in A_n$; we show that $\overline{\{x_n: n< \omega\}}\cong \beta \omega$. It is enough to show that if $S, T\subseteq \omega$ are disjoint, then $\overline{\{x_n: n\in S\}}$ and $\overline{\{x_n: n\in T\}}$ are disjoint. Write $A_S = \bigcup_{n\in S} A_n$, and similarly for $A_T$. Notice that $UA_S\cap UA_T = \emptyset$; this in turn implies that $\mathrm{int}(\overline{UA_S})\cap \mathrm{int}(\overline{UA_T}) = \emptyset$. Towards a contradiction, suppose $x\in \overline{A_S}\cap \overline{A_T}$. Since $\M(G)$ is MHP and $x\in \overline{A_S}$, we see that $x\in \mathrm{int}(\overline{UA_S})$, and likewise for $A_T$. This is a contradiction.
	\end{proof}
	\vspace{2 mm}

	Comparing \Cref{Prop:CAPTukey} and \Cref{Cor:NotCAPEmbedBeta}, it is natural to ask if these conclusions are mutually exclusive, as they are in the case that $G$ is metrizable. We will see later that in general, this need not be the case.

	\Cref{Prop:BadOpensNotCAP} has other applications towards showing that the topological dynamics of groups which are not UEB are badly behaved. For instance, the following result shows that for groups which are not UEB, the operation $G\to \M(G)$ does not respect product.
	\vspace{2 mm}

	\begin{theorem}
		\label{Thm:NotCAPNotProductive}
		Suppose $G$ and $H$ are topological groups, neither of which is UEB. Then $\M(G\times H)\not\cong \M(G)\times \M(H)$.
	\end{theorem}

	\begin{proof}
		We will show that $\M(G)\times \M(H)$ is not an MHP $(G\times H)$-flow, unlike $\M(G \times H)$.
		We need to find an open $W\subseteq \M(G)\times \M(H)$, an identity neighborhood $U\times V\in \cal{N}_{G\times H}$, and $p\in \overline{W}$ so that $p\not\in \mathrm{Int}\left(\overline{(U\times V)\cdot W}\right)$. Using \Cref{Prop:BadOpensNotCAP}, fix a sequence $\{A_n: n< \omega\}$ of non-empty open subsets of $\M(G)$ and $U\in \cal{N}_G$ with $\{UA_n: n<\omega\}$ pairwise disjoint; similarly, fix a sequence $\{B_n: n< \omega\}$ of non-empty open subsets of $\M(H)$ and $V\in \cal{N}_H$ with $\{VB_n: n< \omega\}$ pairwise disjoint. We set $W = \bigcup_n A_n\times B_n$.
		\vspace{2 mm}

		\begin{claim}
			$\mathrm{Int}\left(\overline{(U\times V)\cdot W}\right) = \bigcup_n \mathrm{Int}\left(\overline{UA_n\times VB_n}\right)$
		\end{claim}

		\begin{proof}
			Clearly the right hand side is included in the left. Now let $q\in \mathrm{Int}\left(\overline{(U\times V)\cdot W}\right)$. Let $A\times B\subseteq \overline{(U\times V)\cdot W}$ be a basic open neighborhood of $q$.
			There is some $n< \omega$ such that $UA_n\times VB_n$ meets $A\times B$.
			We claim that $A\times B\subseteq \overline{UA_n\times VB_n}$.
			If not, then without loss of generality we can find non-empty open $C\subseteq A$ with $C\cap UA_n = \emptyset$.
			But now setting $D = B\cap VB_n$, we have for every $m\neq n$ that $(UA_m\times VB_m)\cap (C\times D) = \emptyset$, simply because $VB_m\cap D = \emptyset$.
			Since also $(UA_n\times VB_n)\cap (C\times D) = \emptyset$, it cannot be that  $A\times B\subseteq \overline{(U\times V)\cdot W} = \overline{\bigcup_{n} UA_{n} \times VB_{n}}$.
			This contradiction concludes the proof of the claim.
		\end{proof}

		Now let $p$ be any point in $\bigcap_{N} \overline{\bigcup_{n>N} A_n\times B_n} \subseteq \overline{W}$. Then $p\not\in \mathrm{Int}\left(\overline{UA_n\times VB_n}\right)$ for any $n< \omega$. By the claim, it follows that $\M(G)\times \M(H)$ is not MHP.
	\end{proof}
	\vspace{2 mm}

	It will turn out that \Cref{Thm:NotCAPNotProductive} is an if and only if: we show in \Cref{Prop:ProductFormula} that if either of $G$ or $H$ is UEB, then $\M(G\times H)\cong \M(G)\times \M(H)$.

	\section{The equivalence of UEB and CAP }
	\label{Sec:EquivCAPUEB}
	We are now ready for one of the main theorems of the paper.
	Recall that $\lip(d)$, for a continuous, right-invariant pseudo-metric $d$ on $G$, is a $G$-flow with the action given by $g \cdot f(h) = f(hg)$.
	\vspace{2 mm}

	\begin{theorem}
		\label{Thm:EquivalentInG}
		Fix a topological group $G$. Then the following are equivalent.
		\begin{enumerate}
			\item
			$G$ is CAP.
			\item
			$G$ is strongly CAP.
			\item
			For any continuous, diameter $1$, right-invariant pseudo-metric $d$ on $G$, the set $\ap(\lip(d))\subseteq \lip(d)$ is closed.
			\item
			$G$ is UEB.
		\end{enumerate}
	\end{theorem}

	\begin{rem}
		Recall that an \emph{ambit} is a $G$-flow $X$ along with a distinguished point $x\in X$ with dense orbit. A topological group is said to be \emph{ambitable} if for any continuous, diameter $1$, right-invariant pseudo-metric $d$ on $G$, the flow $\lip(d)$ embeds into some ambit.
		For the class of \emph{ambitable} topological groups, the above are also equivalent to $\ap(\S(G))$ being closed. The same is true for pre-compact groups (which are never ambitable), simply because these are all CAP. It is an open question whether every topological group is either pre-compact or ambitable; see \cite{Pachl}.
	\end{rem}

	\begin{proof}
		We have $(2)\Rightarrow (1)\Rightarrow (3)$. We will show that $(4)\Rightarrow (2)$ and $\neg(4)\Rightarrow \neg(3)$. We will freely use the equivalent characterizations of $(4)$ given in \Cref{Prop:CAPSimpleEquiv}.

		For $(4)\Rightarrow (2)$ we closely follow Section 3 of \cite{JZ}. Fix $\cal{D}$ a base of continuous, diameter $1$, right-invariant pseudo-metrics on $G$ for which item $(4)$ holds, and fix a $G$-flow $X$. Suppose $x_i, y_i\in \ap(X)$ with $E_G(x_i, y_i)$ with $x_i\to x$ and $y_i\to y$. We wish to show that $x, y\in \ap(X)$ and $E_G(x, y)$. Suppose $Y_i\subseteq X_i$ is the minimal subflow with $x_i, y_i\in Y_i$. Let $\phi_i\colon \M(G)\to Y_i$ be a $G$-map, and choose $p_i, q_i\in \M(G)$ with $\phi_i(p_i) = x_i$ and $\phi_i(q_i) = y_i$.

		Suppose in $\M(G)$ that $p_i\to p$. We claim that $\phi_i(p)\to x$. Fix some $A\ni_{op} x$. We must show that eventually $\phi_i(p)\in A$. Find $B\ni_{op} x$ and $U\in \cal{N}_G$ with $\overline{UB}\subseteq A$. Fix a pseudo-metric $d\in \cal{D}$ and $\epsilon > 0$ with $d(\epsilon)\subseteq U$. Since $\partial_d$ is continuous, we have $p_i\xrightarrow{\partial_d} p$. So eventually, we have $\phi_i(p_i)\in B$ and $\partial_d(p_i, p) < \epsilon$, which together imply that $p\in \overline{d(\epsilon)\cdot \phi_i^{-1}(B)}\subseteq \phi_i^{-1}\left(\overline{d(\epsilon)\cdot B}\right)$. So we have $\phi_i(p)\in \overline{d(\epsilon)\cdot B}\subseteq A$ as desired. If we also suppose that $q_i\to q$, an identical argument shows that $\phi_i(q)\to y$.

		Now suppose we are given $r\in \S(G)$ with $rp = q$. We claim that $rx = y$. Let $g_i\in G$ with $g_i\to r$, and let $A\ni_{op} y$. We wish to show that eventually $g_ix \in A$. Find $B\ni_{op} y$ and $U\in \cal{N}_G$ with $\overline{UA}\subseteq B$. Fix a pseudo-metric $d\in \cal{D}$ and $\epsilon > 0$ with $d(\epsilon)\subseteq U$. Since $\partial_d$ is continuous, we have $g_ip\xrightarrow{\partial_d} q$. Fix any index $j$ with $\partial_d(g_jp, q) < \epsilon$; for this $j$ and any index $i$ with  $\phi_i(q)\in B$, we have $g_jp\in \overline{d(\epsilon)\cdot \phi_i^{-1}(B)}\subseteq \phi_i^{-1}\left(\overline{d(\epsilon)\cdot B}\right)$. So $\phi_i(g_jp) = g_j\phi_i(p)\in \overline{d(\epsilon)\cdot B}$. So $g_jx\in \overline{d(\epsilon)\cdot B}\subseteq A$ as desired. A symmetric argument shows that if we are given $s\in \S(G)$ with $p = sq$, then also $x = sy$.

		To conclude the argument, we note that for any minimal subflow $M\subseteq \S(G)$ we have $Mp = Mq = \M(G)$ by \Cref{Fact:SGSemigroup}, so we can find $r, s\in M$ with $rp = q$ and $p = sq$. Hence $y\in \hat{\rho}_x[M]$ and $x\in \hat{\rho}_y[M]$. This implies that $x,y\in \ap(X)$ and that both $x\in \overline{Gy}$ and $y\in \overline{Gx}$, i.e.\ that $E_G(x, y)$ as desired.
		\vspace{5 mm}

		For $\neg(4)\Rightarrow \neg(3)$, we follow section 2.7 of \cite{ZucThesis}. Suppose that the UEB uniformity on $\M(G)$ is strictly finer than the compact topology. By Lemma~\ref{Lem:ANWDBall}, we can find a continuous, bounded, right-invariant pseudo-metric $d$ on $G$ and a point $p\in \M(G)$ so that, setting $\partial := \partial_d$, the metric ball $B_\partial(p, 1) := \{q\in \M(G): \partial(p, q)< 1\}$ is $\tau$-nowhere dense. \textbf{Fix this} $p\in \M(G)$.

		\iffalse
		$\partial_d \coloneqq \partial$ is not continuous.
		We will show (after some minor modifications to $d$) that $\ap(\lip(d))\subseteq \lip(d)$ is not closed. Since $\partial$ is $\tau$-lsc and adequate by Propositions~\ref{Prop:TauLSC} and \ref{Prop:Adequacy}, there is $p\in \M(G)$ and some $\epsilon > 0$ with $B_\partial(p, \epsilon) \coloneqq \{q\in \M(G): \partial(p, q)< \epsilon\}$ $\tau$-nowhere dense.
		To see why, it is easier to work with the closed $\partial$-ball $C_\partial(p, \epsilon) \coloneqq \{q\in \M(G): \partial(p, q)\leq \epsilon\}$.
		Since $\partial$ is $\tau$-lsc, $C_\partial(p, \epsilon)$ is $\tau$-closed. So suppose towards a contradiction that for every $p\in \M(G)$ and $\epsilon > 0$ that $\mathrm{Int}\left(C_\partial(p, \epsilon)\right)\neq \emptyset$. Using adequacy, we see that $p\in \mathrm{Int}\left(B_\partial(p, 2\epsilon)\right)$. But this implies that $\partial$ is continuous, contradicting our assumption.
		So we may find $p\in \M(G)$ and $\epsilon > 0$ with $\mathrm{Int}\left(C_\partial(p, \epsilon)\right) = \emptyset$. By multiplying the pseudo-metric $d$ by $1/\epsilon$ and then capping at $1$, we may assume that $B_\partial(p, 1)$ is $\tau$-nowhere dense. \textbf{Fix this} $p\in \M(G)$.
		\fi

		For any continuous $\psi\colon \M(G)\to [0, 1]$, we define $\psi_\partial\colon \M(G)\to [0,1]$ via
		$$\psi_\partial(q) = \inf\{\psi(r) + \partial(q, r): r\in \M(G)\}.$$

		\begin{lemma}
			\label{Lem:PsiPartialCont}
			The function $\psi_\partial$ is continuous and $\partial$-Lipschitz.
		\end{lemma}

		\begin{proof}
			Suppose $q_i\to q$. Let $\epsilon > 0$, and choose $r_i\in \M(G)$ such that $\psi(r_i) + \partial(q_i, r_i) \leq \psi_\partial(q_i)+\epsilon$. Letting $r_i\to r$, we have $\psi(r_i)\to \psi(r)$, and eventually $\partial(q, r) \leq \partial(q_i, r_i)+ \epsilon$ by $\tau$-lsc. It follows that eventually
			\begin{align*}
			\psi_\partial(q) &\leq \psi(r)+\partial(q, r)\\[1 mm]
			&\leq \psi(r_i) + \partial(q_i, r_i) + 2\epsilon\\[1 mm]
			&\leq \psi_\partial(q_i) + 3\epsilon.
			\end{align*}
			For the other inequality, choose $s\in \M(G)$ with $\psi(s)+\partial(q, s)\leq \psi_\partial(q)+\epsilon$. Suppose $c> 0$ is such that $c - \epsilon < \partial(q, s) < c$.
			By continuity of $\psi$, $A\ni_{op} s$ can be chosen so that $\psi[A]\subseteq [\psi(s)-\epsilon, \psi(s)+\epsilon]$. Then $B_\partial(A, c)$ is open by adequacy and contains $q$, as $\partial(q, s) < c$ and $s \in A$.
			So eventually $q_i\in B_\partial(A, c)$, and we can choose witnesses $s_i\in A$ with $\partial(q_i, s_i) < c$. It follows that eventually
			\begin{align*}
			\psi_\partial(q) &\geq \psi(s)+\partial(q, s) - \epsilon\\[1 mm]
			&\geq \psi(s_i) + \partial(q_i, s_i) - 3\epsilon\\[1 mm]
			&\geq \psi_\partial(q_i)-3\epsilon.
			\end{align*}
			To see that $\psi_{\partial}$ is $\partial$-Lipschitz, let $q, r\in \M(G)$. Fix $\epsilon > 0$, and suppose $s\in \M(G)$ satisfies $\psi_\partial(q)\geq \psi(s)+\partial(q, s) -\epsilon$. Then we have
			\begin{align*}
			\psi_{\partial}(r) &\leq \psi(s)+ \partial(r, s)\\[1 mm]
			&\leq \psi(s)+ \partial(r, q) + \partial(q, s)\\[1 mm]
			&\leq \psi_{\partial}(q) + \partial(r, q) + \epsilon.
			\end{align*}
			Since $\epsilon> 0$ is arbitrary and applying an identical argument with $p$ and $q$ reversed, we are done.
		\end{proof}
		\vspace{2 mm}

		Now given $q\in \M(G)$, define $q_\psi\colon G\to [0,1]$ via $q_\psi(g) = \psi_\partial(g\cdot q)$. We note that $q_\psi\in \lip(d)$ since $\partial(gq, hq) \leq d(g, h)$. Even better, the map from $\M(G)$ to $\lip(d)$ given by $q\to q_\psi$ is a $G$-map, for $h \cdot q_{\psi}(g) = q_{\psi}(g \cdot h) = \psi_{\partial}(g \cdot hq)= (hq)_\psi(g)$, so we have $q_\psi\in \ap\left(\lip(d)\right)$.

		Recall the $p\in \M(G)$ we fixed earlier. \textbf{Define} $f\colon G \to [0, 1]$ via
		$$f(g) = \partial(gp, p).$$
		Since $\lvert \partial(gp, p) - \partial(hp, p) \rvert \leq  \partial(gp, hp) \le d(g, h)$, we have that $f \in \lip(d)$.

		For each $A\ni_{op} p$, fix some $\psi_A\colon \M(G)\to [0,1]$ with $\psi_A(p) = 0$ and with constant value $1$ outside $A$. Set $f_A = p_{\psi_A}$. So we saw above that $f_A\in \ap(\lip(d))$.
		\vspace{2 mm}

		\begin{lemma}
			\label{Lem:FunctionInClosureAP}
			$f\in \overline{\{f_A: A\ni_{op} p\}}.$
		\end{lemma}

		\begin{proof}
			First note for any $A\ni_{op} p$ and any $g\in G$ that $f(g)\geq f_A(g)$, simply because $f_A(g) = \inf\{\psi_A(r) + \partial(gp, r): r\in \M(G)\}$, and one can consider $r = p$.
			Now fix some finite $F\subseteq G$ and $\epsilon >0$;
			we want to find $A\ni_{op} p$ so that $f(g)\leq f_A(g)+\epsilon$ for every $g\in F$, that is, such that $f(g) \leq \psi_A(r)+\partial(gp, r)+\epsilon$ for every $g\in F$ and $r\in \M(G)$.
			We can suppose that $\epsilon < \partial(g, gp)$ for all $g \in F$, and write $c_g = \partial(g, gp) - \epsilon$.
			For each $g \in F$, there is $A_g \ni_{op} p $ such that $pg \not \in \overline{d(c_g)A}$.
			Let $A = \bigcap_{g \in F} A_g$.
			Fix $g \in F$; since $pg \not \in \overline{d(c_g)A}$, if $r \in A$, we have $\partial(pg, r) \ge c_g = \partial(g, gp) - \varepsilon$, that is,
			$f(g) = \partial(gp, p)\leq \partial(gp, r)+\epsilon$.
			If $r \not \in A$, then $f(g) \le \psi_A(r) = 1$, so $A$ is as desired.
		\end{proof}
		\vspace{2 mm}

		A set $S\subseteq G$ is \emph{syndetic} if there is a finite $F\subseteq G$ with $G = FS$. We need the following lemma.
		\vspace{2 mm}

		\begin{lemma}
			\label{Lem:SyndeticPreimage}
			Suppose $\xi\in \ap(\lip(d))$. Then for any open $U\subseteq [0,1]$ the set $\xi^{-1}(U)\subseteq G$ is either empty or syndetic.
		\end{lemma}

		\begin{proof}
			Suppose $\xi^{-1}(U)$ is non-empty, but not syndetic. We can find for every finite $F\subseteq G$ some $g_F\in G$ with $(g_F\cdot \xi)(h) = \xi(hg_F)\not\in U$ for every $h\in F$.
			Consider the net $g_F \xi$, indexed by the finite subsets of $G$.
			Up to a subnet it converges to some $\phi$.
			Let $h \in G$ and suppose towards a contradiction that $\phi(h) \in U$.
			Eventually $g_F \xi(h)= \xi(h g_F) \in U$, but for each such $F$, $g_{F \cup \set{h}} \xi(h) \not \in U$, a contradiction.
			So $\ran(\phi)\cap U = \emptyset$.
			On the other hand, if $g_i \phi \to \xi$ and $\xi(h) \in U$, eventually $g_i \phi(h) = \phi(h g_i) \in U$, which cannot be, so $\xi \not \in \overline{G\cdot \phi}$.
			Therefore  $\xi$ does not belong to a minimal subflow, i.e.\ $\xi\not\in \ap(\lip(d))$.
		\end{proof}
		\vspace{2 mm}

		Considering the $f$ defined before, we show that $f\not\in \ap(\lip(d))$. Towards a contradiction, suppose it were. By \Cref{Lem:SyndeticPreimage}, $S \coloneqq f^{-1}([0, 1/2))\subseteq G$ is syndetic, so for some finite $F\subseteq G$, we have $G = FS$. Since $G\cdot p\subseteq \M(G)$ is dense, we have that $S\cdot p\subseteq \M(G)$ is somewhere dense. However, $S\cdot p\subseteq B_\partial(p, 1)$, and must be nowhere dense. This contradiction shows that $f\not\in \ap(\lip(d))$.

		This concludes the proof of $\neg(4)\Rightarrow \neg(3)$ and the proof of \Cref{Thm:EquivalentInG}.
	\end{proof}
	\vspace{2 mm}

	As we have now seen that UEB groups and CAP groups are the same, we can use the names interchangeably; we will typically use CAP unless we wish to emphasize \Cref{Def:UEBGroup}.

	\section{Closure properties of CAP groups}
	\label{Sec:ClosureProps}

	The class of CAP groups enjoys robust closure properties, which we collect in this section. Some are simple observations, while one generalizes a non-trivial results of \cite{JZ}. We will also be able to complete our analysis of when $G\to \M(G)$ respects product, obtaining yet another characterization of CAP groups.
	\vspace{2 mm}

	\begin{prop}
		\label{Prop:CAPDenseHom}
		Let $G$ and $K$ be topological groups. Suppose $G$ is CAP and that $\pi\colon G\to K$ is a continuous homomorphism with dense image. Then $K$ is CAP.
	\end{prop}

	\begin{proof}
		This follows since every $K$-flow $X$ becomes a $G$-flow via the action $g\cdot x \coloneqq \pi(g)\cdot x$. Since $\im(\pi)\subseteq K$ is dense, $G$ orbit closures and $K$ orbit closures coincide, so we have that $\ap_K(X) = \ap_G(X)$. Since $G$ is CAP, $K$ is as well.
	\end{proof}
	\vspace{2 mm}

	The next proposition shows that CAP groups are closed under surjective inverse limits; first we give a few general remarks about inverse limits of topological groups. An \emph{inverse system} of topological groups, denoted $(G_i, \pi^j_i)$, consists of a directed set $I$, topological groups $G_i$ for each $i\in I$, and a continuous group homomorphism $\pi^j_i\colon G_j\to G_i$ for each $i\leq j\in I$. The \emph{inverse limit} of the inverse system $(G_i, \pi^j_i)$ is the topological group
	$$\varprojlim_i G_i \coloneqq \setnew{(g_i)_i\in \prod_i G_i}{ \forall i\leq j\in I\, \pi^j_i(g_j) = g_i}.$$
	The notation avoids mentioning the maps $\pi^j_i$ which are usually understood from context. Notice that $\varprojlim_i G_i$ could be empty. Let $\pi_i\colon \varprojlim_i G_i\to G_i$ be the projection to coordinate $i$. We say that the inverse system is \emph{surjective} if each $\pi_i$ is surjective. In particular, the inverse limit of a surjective inverse system is non-empty. Note that this is stronger than demanding that each $\pi^j_i$ be surjective.

	Notice that each bonding homomorphism $\pi^j_i$ continuously extends to a surjective map $\hat{\pi}^j_i\colon \S(G_j)\to \S(G_i)$.
	\vspace{2 mm}

	\begin{lemma}
		\label{Lem:InverseSG}
		Suppose $(G_i, \pi^j_i)$ is a surjective inverse system of topological groups, and let $G = \varprojlim_i G_i$. Then $\S(G) \cong \varprojlim_i \S(G_i)$.
	\end{lemma}

	\begin{proof}
		Viewing $\S(G)$ and $\S(G_i)$ as spaces of near ultrafilters, let $\phi\colon \S(G)\to \varprojlim_i \S(G_i)$ denote the map
		$p \mapsto \left ( \setnew{A \subseteq G_i}{\pi_i^{-1}(A) \in p}\right)_{i \in I}$.
		This is a continuous and surjective $G$-map, and we argue that it is injective. Suppose $p\neq q\in \S(G)$. Find $A\in p$, $B\in q$, and $U\in \cal{N}_G$ with $UA\cap UB = \emptyset$. We may assume that $U = \{g\in G: \pi_i(g) \in U_i\}$ for some $i \in I$ and  $U_i\in \cal{N}_{G_i}$. In particular, the sets $UA$ and $UB$ are $\pi_i$-invariant, that is $\pi_i^{-1}(\pi_i[UA]) = UA$, and similarly for $UB$.
		Therefore we have $\pi_i[A]\in p_i$, $\pi_i[B]\in q_i$, and $U_i\pi_i[A]\cap U_i\pi_i[B] = \pi_i[UA]\cap \pi_i[UB] = \emptyset$, i.e.\ that $p_i\neq q_i$.
	\end{proof}
	\vspace{0 mm}

	\begin{prop}
		\label{Prop:CAPInverseLim}
		Suppose $(G_i, \pi^j_i)$ is a surjective inverse system of topological groups, and let $G = \varprojlim_i G_i$. If every $G_i$ is CAP, then $G$ is CAP.
	\end{prop}

	\begin{proof}
		By \Cref{Lem:InverseSG}, we have $\S(G) \cong \varprojlim_i \S(G_i)$. If $M\subseteq \S(G)$ is a minimal subflow and $i\in I$, then $\pi_i[M] =: M_i\subseteq \S(G_i)$ is a minimal $G_i$-subflow. A base of continuous, diameter $1$, right-invariant pseudo-metrics on $G$ is given by those of the form $d\circ (\pi_i\times \pi_i)$ for some $i\in I$ and some continuous, diameter $1$, right-invariant pseudo-metric $d$ on $G_i$. If $f\in \lip(d\circ (\pi_i\times \pi_i))$, and $g, h \in G$ are such that $\pi_i(g) = \pi_i(h)$, then $\lvert f(g)-f(h) \rvert \le d(\pi_i(g), \pi_i(h)) = 0$, so $f(g)=f(h)$.
		Therefore $f$ factors through $G_i$: there is $f' \in \lip(d)$ with $f= f'\circ \pi_i$.
		It follows that $\partial_{d\circ (\pi_i\times \pi_i)} = \partial_d\circ (\pi_i\times \pi_i)$. By assumption, $\partial_d$ is continuous on $M_i$, so $\partial_{d\circ (\pi_i\times \pi_i)}$ is continuous on $M$ as desired.
	\end{proof}
	\vspace{2 mm}

	We discuss one last general result about inverse systems which will be useful later. Suppose $(G_i, \pi^j_i)$ is a surjective inverse system of topological groups. If $X_i$ is a $G_i$-flow, then we can also view $X_i$ as a $G_j$-flow for any $j\geq i$. An \emph{inverse system of $G_i$-flows}, denoted $(X_i, \phi^j_i)$, consists of a $G_i$-flow $X_i$ for each $i\in I$ along with a $G_j$-map $\phi^j_i\colon G_j\to G_i$ for each $i\leq j\in I$. Then $\varprojlim_i X_i$ becomes a $\varprojlim_i G_i$-flow in the natural way.

	We saw in the proof of \Cref{Prop:CAPInverseLim} that $\M(G) = \varprojlim_i \M(G_i)$ for some inverse limit. We show that any inverse limit will work.
	\vspace{2 mm}

	\begin{lemma}
		\label{Lem:InverseMG}
		Suppose $(G_i, \pi^j_i)$ is a surjective inverse system of topological groups. Suppose $(X_i, \phi^j_i)$ is an inverse system of $G_i$-flows so that $X_i\cong \M(G_i)$ for every $i\in I$. Then letting $X = \varprojlim_i X_i$, we have $X\cong \M(G)$.
	\end{lemma}

	\begin{proof}
		Fix a minimal subflow $M\subseteq \S(G)$, and set $\pi_i[M] = M_i\subseteq \S(G_i)$. Fix $x\in X$. This gives us a $G$-map $\hat{\rho}_x\colon M\to X$ given by $\hat{\rho}_x(p) = px$. Writing $p = (p_i)_{i\in I}$ and $x = (x_i)_{i\in I}$, we have $px = (p_ix_i)_{i\in I}$. If $p\neq q\in \S(G)$, then $p_i\neq q_i$ for some $i\in I$. Since $X_i\cong \M(G_i)$, we have that $\hat{\rho}_{x_i}\colon M_i\to X_i$ is an isomorphism, implying that $p_ix_i\neq q_ix_i$. Hence $\hat{\rho}_x$ is injective and therefore an isomorphism.
	\end{proof}
	\vspace{2 mm}

	We now generalize a theorem from \cite{JZ}, where it is shown that given a short exact sequence of Polish groups $1\to H\to G\to K\to 1$ so that $\M(H)$ and $\M(K)$ are metrizable, then $\M(G)$ is also metrizable. An important step in that proof was to show that Polish groups $G$ with $\M(G)$ metrizable are CAP. Here, we use the strong CAP property directly to show that CAP groups are closed under extensions.
	\vspace{2 mm}

	\begin{theorem}
		\label{Thm:CAPShortExactSeq}
		Suppose $1\to H\to G\xrightarrow{\pi} K\to 1$ is a short exact sequence of topological groups. Then if $H$ and $K$ are CAP, then so is $G$.
	\end{theorem}

	\begin{proof}
		We break the proof of this theorem into several claims.
		\vspace{2 mm}

		\begin{claim}
			Suppose $Y$ is a $G$-flow. Then $\ap_H(Y)$ is a $G$-subflow. If $Y$ is a minimal $G$-flow, then $\ap_H(Y) = Y$.
		\end{claim}

		\begin{proof}
			Suppose $Z\subseteq Y$ is a minimal $H$-subflow. Then so is $gZ$, since $H$ is a normal subgroup. Hence $\ap_H(Y)$ is $G$-invariant, and it is closed since $H$ is CAP. If $Y$ is a minimal $G$-flow, then $\ap_H(Y)\subseteq Y$ is dense, and it is closed since $H$ is CAP.
		\end{proof}
		\vspace{1 mm}

		\begin{claim}
			Suppose $Y$ is a $G$-flow with $\ap_H(Y) = Y$. Then $Y/E_H$ is a $K$-flow, and the quotient map $\pi_H\colon Y\to Y/E_H$ is open.
		\end{claim}

		\begin{proof}
			Since $H$ is CAP, $Y/E_H$ is a compact Hausdorff space. Since the $G$-action on $Y$ is continuous and preserves the closed equivalence relation $E_H$, we obtain a continuous $G$-action on $Y/E_H$. Since $H$ acts trivially on $Y/E_H$, we can view $Y/E_H$ as a $K$-flow. To see that $\pi_H$ is open, suppose $A\subseteq Y$ is open. Then $\pi_H[A] = \pi_H[HA]$, and $HA$ is open and $E_H$-invariant.
		\end{proof}
		\vspace{2 mm}

		Now suppose $X$ is a $G$-flow, towards showing that $\ap_G(X)$ is closed. By the first claim, we have $\ap_G(X)\subseteq \ap_H(X)$. By the second claim, $\ap_H(X)/E_H \coloneqq W$ is a $K$-flow. We will argue that
		$$\ap_G(X) = \pi_H^{-1}(\ap_K(W)).$$
		Since $K$ is CAP, this will show that $\ap_G(X)$ is closed. In one direction, suppose $Y\subseteq X$ is a minimal $G$-subflow. Then $\pi_H[Y]\subseteq W$ is a minimal $G$-subflow, so also a minimal $K$-subflow. In the other direction, suppose $Z\subseteq \ap_K(W)$ is a minimal $K$-subflow. Then $\pi_H^{-1}(Z)\subseteq \ap_G(X)$ is a $G$-subflow, and we argue that it is minimal. Fix $p\in \pi_H^{-1}(Z)$ and a non-empty open $A\subseteq \pi_H^{-1}(Z)$. By the second claim, we have that $\pi_H[A]\subseteq Z$ is open, so find $g\in G$ with $g\cdot \pi_H(p)\in \pi_H[A]$. This means that $gp\in Q$ for some minimal $H$-subflow $Q\subseteq \pi_H^{-1}(Z)$ with $Q\cap A\neq \emptyset$. We then find $h\in H$ with $hgp\in A$.
	\end{proof}
	\vspace{0 mm}

	\begin{cor}
		\label{Cor:CAPProducts}
		If $\{G_i: i\in I\}$ is a set of CAP groups, then $\prod_i G_i$ is CAP.
	\end{cor}

	\begin{proof}
		Combine \Cref{Prop:CAPInverseLim} and \Cref{Thm:CAPShortExactSeq}.
	\end{proof}
	\vspace{2 mm}

	In fact, when $G = \prod_i G_i$ for a collection $\{G_i: i\in I\}$ of CAP groups, we can describe $\M(G)$ explicitly.
	\vspace{2 mm}

	\begin{prop}
		\label{Prop:ProductFormula}
		Suppose $\{G_i: i\in I\}$ is a set of CAP groups. Then $M\left(\prod_i G_i\right) \cong \prod_i \M(G_i)$. More generally, if $H$ and $K$ are topological groups with $H$ CAP, then \newline $\M(H\times K) = \M(H)\times \M(K)$.
	\end{prop}

	\begin{proof}
		By \Cref{Lem:InverseMG}, we have $\M\left(\prod_i G_i\right)\cong \varprojlim \M(G_F)$, where $F\subseteq I$ is finite and $G_F = \prod_{i\in F} G_i$, so it is enough to show the second claim. Set $G = H\times K$. By the claims in the proof of \Cref{Thm:CAPShortExactSeq}, we have that $\ap_H(\M(G)) = \M(G)$ and that $\M(G)/E_H$ is a $K$-flow. Because $\M(G)$ is a minimal $G$-flow, $\M(G)/E_H$ is a minimal $K$-flow. If $X$ is any minimal $K$-flow, then $X$ is also a $G$-flow, so let $\phi\colon \M(G)\to X$ be a $G$-map. Since $H$ acts trivially on $X$, $\phi$ must be $E_H$-invariant, giving us a $K$-map $\tilde{\phi}\colon \M(G)/E_H\to X$. Hence we have $\M(K)\cong \M(G)/E_H$. Let $\pi\colon \M(G)\to \M(K)$ denote the quotient map.

		The homomorphism from $G$ onto $H$ also induces an $H$-map $\psi\colon \M(G)\to \M(H)$. Notice that $\psi$ maps each $H$-minimal subflow of $\M(G)$ onto $\M(H)$, hence each $H$-minimal subflow of $\M(G)$ is isomorphic to $\M(H)$ via $\psi$. Form the $G$-map $\theta\colon \M(G)\to \M(H)\times \M(K)$. We show that $\theta$ is injective, and hence an isomorphism. Fix $p\neq q\in \M(G)$; if $\pi(p)\neq \pi(q)$, we are done. If $\pi(p) = \pi(q)$, then $p$ and $q$ belong to the same $H$-minimal subflow of $\M(G)$. Since $p\neq q$, we must have $\psi(p)\neq \psi(q)$.
	\end{proof}
	\vspace{0 mm}

	\begin{cor}
		\label{Cor:CAPviaProducts}
		Let $G$ be a topological group. Then the following are equivalent:
		\begin{enumerate}
			\item
			$G$ is CAP.
			\item
			$\M(G\times G)\cong \M(G)\times \M(G)$.
		\end{enumerate}
	\end{cor}

	\begin{proof}
		Combine \Cref{Prop:ProductFormula} and \Cref{Thm:NotCAPNotProductive}.
	\end{proof}
	\vspace{2 mm}

	We conclude this section with two examples illustrating the possible overlap between the conclusions of \Cref{Prop:CAPTukey} and \Cref{Cor:NotCAPEmbedBeta}.
	\vspace{2 mm}

	\begin{exa}
		Let $G = 2^\mathfrak{c}$. Then $G$ is CAP since it is a compact group. In this case, $\M(G) = G = 2^\mathfrak{c}$, which embeds a copy of $\beta \omega$.
	\end{exa}
	\vspace{0 mm}

	\begin{exa}
		Let $G = \bb{Z}\times 2^\mathfrak{c}$. By \Cref{Prop:CAPDenseHom}, $G$ cannot be CAP since $\bb{Z}$ is not CAP. However, since $2^\mathfrak{c}$ is CAP, we have by  \Cref{Prop:ProductFormula} that
		$$\M(G) \cong \M(\bb{Z})\times \M(2^\mathfrak{c}) = \M(\bb{Z})\times 2^\mathfrak{c}.$$
		Therefore $\M(G)$ has a clopen basis of size continuum. On the other hand, the poset $\cal{N}_G$ has a cofinal subset isomorphic to  the poset $[\mathfrak{c}]^{<\omega}$. The poset $\mathrm{Nbd}(\Delta_{\M(G)})$ has a cofinal subset of size continuum, the clopen neighborhoods, which is a \emph{join-semilattice}, i.e.\ finite sets have least upper bounds. We conclude by noting that $[\mathfrak{c}]^{<\omega}$ admits a monotone, cofinal map to any join-semilattice $\bb{P}$ of size continuum; simply put $\mathfrak{c}$ and $\bb{P}$ in bijection, and map a finite subset of $\mathfrak{c}$ to the corresponding least upper bound.
	\end{exa}
	\vspace{0 mm}

	\begin{que}
		\label{Que:TopologyMGCAP}
		Is there a topological condition on $\M(G)$ which characterizes when $G$ is CAP?
	\end{que}

	\section{Groups with large CAP subgroups}
	\label{Sec:LargeCAPSubs}

	Recall that if $G$ is a topological group and $H\subseteq G$ is a closed subgroup, then $G/H$ is also a Hausdorff uniform space, where the typical basic open entourage is of the form $\{(gH, kH): gHk^{-1}\cap U\neq \emptyset\}$ for some $U\in \cal{N}_G$. Therefore we can form the completion $\widehat{G/H}$ and the Samuel compactification $\S(G/H)$ of this uniform space. As usual, we have $\widehat{G/H}\subseteq \S(G/H)$. The action of $G$ on $G/H$ extends, turning $\S(G/H)$ into a $G$-flow. We say that $H\subseteq G$ is \emph{co-precompact} if $\widehat{G/H} = \S(G/H)$, i.e.\ if $\widehat{G/H}$ is compact. Equivalently, for every $U\in \cal{N}_G$, there is a finite $F\subseteq G$ with $UFH = G$.
	\vspace{2 mm}

	\begin{defin}
		\label{Def:PreSyndetic}
		A subset $S \subseteq G$ is \emph{pre-syndetic} if,
		for every $U \in \cal{N}_{G}$, there is a finite $F\subseteq G$ with $FUS = G$.
	\end{defin}
\vspace{2 mm}

	By \cite{ZucMHP}*{Proposition 6.6} a closed subgroup $H$ of $G$ is pre-syndetic if and only if $\S(G/H)$ is a minimal $G$-flow.

	In this section we consider some generalizations of a result from \cite{MNT}. There, it is shown that if $G$ is a Polish group and $H\subseteq G$ is an extremely amenable, co-precompact, pre-syndetic subgroup, then $\M(G) \cong \widehat{G/H}$. So in particular, $\M(G)$ is metrizable. In our setting, we are asking about conditions on $H$ which ensure that $G$ is CAP. We present two such sufficient conditions, both of which use the following lemma. Given a topological group $G$, a $G$-flow $X$, and a closed subgroup $H\subseteq G$, we let $\mathrm{Fix}_H(X)$ denote the collection of $H$-fixed points in $X$. If $y\in \mathrm{Fix}_H(X)$, there is a unique $G$-map $\phi_y\colon \S(G/H)\to X$ with $\phi_y(H) = y$.
	By restricting to $\widehat{G/H} \subseteq \S(G/H)$ we also obtain joint continuity of $\phi_y$, for $y$ varying in $\mathrm{Fix}_H(X)$.
	\vspace{2 mm}

	\begin{lemma}
		\label{Lem:JointCont}
		The map $\widehat{G/H}\times \mathrm{Fix}_H(X)\to X$ given by $(\eta, y)\to \phi_y(\eta)$ is continuous.
	\end{lemma}

	\begin{proof}
		Suppose $g_iH\to \eta\in \widehat{G/H}$ and $y_i\to y\in \mathrm{Fix}_H(X)$. Write $\phi_i$ for $\phi_{y_i}$; we need to show that $\phi_i(g_iH)\to \phi_y(\eta)$. Fix an open $A\ni \phi_y(\eta)$. Find open $B\ni \phi_y(\eta)$ and $U \in \cal{N}_{G}$ with $UB\subseteq A$.
		For all large enough $i$ and $j$ we have $g_jH\in Ug_iH$ (viewing the latter as a set of cosets). Also, $\phi_y(g_iH)\in B$ for any large enough $i$. Fix a suitably large $i$. Then $\phi_j(g_iH) = g_iy_j\in B$ for any suitably large $j$. Noting that $g_jHg_i^{-1}\cap U\neq \emptyset$, fix some $h\in H$ with $g_jhg_i^{-1}\in U$. Then $(g_jhg_i^{-1})(g_iy_j) \in A$ and is equal to $g_jy_j = \phi_j(g_jH)$.
	\end{proof}
	\vspace{0 mm}

	\begin{prop}
		\label{Prop:IfMGisCompletionGHthenHisEA}
		Let $G$ be a topological group and $H$ be a closed subgroup such that $\M(G) \cong \widehat{G/H}$.
		Then $H$ is extremely amenable.
	\end{prop}
	\begin{proof}
		Notice that $ \widehat{G/H}$ is compact so $ \widehat{G/H} = \S(G/H)$.
		Let $M \subseteq \S(G)$ be a minimal $G$-flow.
		By \cite{ZucMHP}*{Proposition 6.4} there is a canonical $G$-map $\pi\colon \S(G) \to \S(G/H)$ such that $\pi^{-1}(H) = \S(H) \subseteq \S(G)$.
		Since $\S(G/H) \cong M(G)$, it follows that $\pi \vert M$ is an isomorphism onto $\S(G/H)$, so $\S(H) \cap M = \pi^{-1}(H) \cap M = \{p\}$ is a singleton and is an $H$-flow since $\pi(h \cdot p) = h \cdot \pi(p) = H$, for any $h \in H$, and $M$ is $H$-invariant.
		Therefore it a minimal $H$-subflow of $S(H)$, so it is isomorphic to $\M(H)$.
	\end{proof}
	\vspace{0 mm}

	\begin{prop}
		\label{Prop:CoPrecompactEASubgroup}
		Suppose $G$ is a topological group containing a closed, co-precompact, extremely amenable subgroup $H$. Then $G$ is CAP. If $H$ is also pre-syndetic, then $\widehat{G/H} \cong \M(G)$.
	\end{prop}

	\begin{proof}
		Let $X$ be a $G$-flow, towards showing that $\ap_G(X)$ is closed. By co-precompactness of $H$, we have $\S(G/H) = \widehat{G/H}$. Let $x_i\to x$ with each $x_i\in \ap_G(X)$. Set $Y_i = \overline{G\cdot x_i}$. Since $H$ is extremely amenable, let $y_i\in Y_i$ be an $H$-fixed point, and write $\phi_i$ for $\phi_{y_i}$. Assume that $y_i\to y\in \mathrm{Fix}_H(X)$. Let $W\subseteq \widehat{G/H}$ be a minimal $G$-subflow. Since we must have $\phi_i[W] = Y_i$, we can find $z_i\in W$ with $\phi_i(z_i) = x_i$. We may assume that $z_i\to z\in W$. By \Cref{Lem:JointCont}, we have $\phi_y(z) = x$. In particular, $x\in \phi_y[W]$, so since $W$ is minimal, we have $x\in \ap_G(X)$.

		Since $H$ is extremely amenable, fix some $p\in \mathrm{Fix}_H(\M(G))$. Forming $\phi_p\colon \widehat{G/H}\to \M(G)$, we see that if $H$ is pre-syndetic, then the domain is minimal, so $\phi_p$ must be an isomorphism.
	\end{proof}
	\vspace{0 mm}

		By combining the results of \cite{MNT} and \cite{BYMT}, one obtains the converse of the above proposition in the case that $G$ is Polish; namely, $\M(G)$ is metrizable  if and only if there is some extremely amenable closed subgroup $H\subseteq G$ with $\M(G) = \widehat{G/H}$. 
		In a first version of this paper, we asked if the same could hold true for CAP groups. 
		Gheysens, in a private communication, provided a counterexample: the group of finitely supported permutations of a countable set, with the topology it inherits as a subgroup of $\mathrm{Sym}(\omega)$, does not contain non-trivial extremely amenable subgroups but is CAP, since its two-sided completion $\mathrm{Sym}(\omega)$ is CAP. 
		The question still remains open for Raikov complete groups, that is, groups which are complete with respect to the two-sided uniformity. Equivalently, a topological group $G$ is Raikov complete if whenever $G'\supseteq G$ is a topological group containing $G$ as a dense subgroup, then $G = G'$. 
		\vspace{2 mm}

	\begin{que}
		\label{Que:MNTConverse}
		If $G$ is a Raikov complete CAP group, is $\M(G) = \widehat{G/H}$ for some extremely amenable closed subgroup $H \subseteq G$? 
		
		On the opposite end, can a non-compact, Raikov complete CAP group act freely on a compact space?
	\end{que}

	\vspace{2 mm}
	We will add to the discussion of \Cref{Que:MNTConverse} after examining the case of automorphism groups of structures. But given that we do not know if this converse holds, perhaps the assumptions of \Cref{Prop:CoPrecompactEASubgroup} can be weakened to only demand that $H$ is CAP instead of extremely amenable. While we are unable to prove this, we can prove the following; the assumption on $H$ is weakened to CAP, but then we must assume that $H$ is also pre-syndetic.
	\vspace{2 mm}

	\begin{prop}
		\label{Prop:CAPCoPrePreSyndSubgroup}
		Suppose $G$ is a topological group containing a closed, co-precompact, pre-syndetic CAP subgroup $H$. Then $G$ is CAP.
	\end{prop}
\vspace{2 mm}

	The proof of \Cref{Prop:CAPCoPrePreSyndSubgroup} follows a similar structure to that of \Cref{Prop:CoPrecompactEASubgroup}.
	Since $H$ is not extremely amenable, we might not find $H$-fixed points in each $G$-flow $X$, so we pass to a related flow which is guaranteed to have $H$-fixed points.
	Given a $G$-flow $X$, let $K(X)$ denote the space of compact subsets of $X$ with the Vietoris topology.
	This is the topology generated by open sets of the form $\mathrm{Sub}(A) \coloneqq \{Z\in K(X): Z\subseteq A\}$ and $\mathrm{Meets}(A) \coloneqq \{Z\in K(X): Z\cap A\neq \emptyset\}$ for some open $A\subseteq X$.
	The space $K(X)$ is compact and forms a $G$-flow with the obvious action.
	We briefly recall the ``circle" operation on the Vietoris hyperflow; if $B\in K(X)$ and $p\in \S(G)$, we write $p\circ B$ for $\lim_i g_iB$, where $g_i\to p$. If we write $pB \coloneqq \{pb: b\in B\}$, then $pB\subseteq p\circ B$, but in general the inclusion is strict.
	Indeed, we have $w\in p\circ B$  if and only if there are nets $g_i\in G$ and $b_i\in B$ with $g_i\to p$ and $g_ib_i\to w$.
	We can restrict the net $b_i$ to come from any desired dense subset of $B$.

	If $H$ is a subgroup of $G$, any $H$-subflow $Y \subseteq X$ is an $H$-fixed point of $K(X)$.
	\vspace{2 mm}

	\begin{lemma}
		\label{Lem:IfCAPthenMinClosed}
		Let $H$ be a closed CAP subgroup of $G$ and $X$ be a $G$-flow.
		Then the set
		$$\mathrm{Min}_H(X)\coloneqq \{Z\in K(X): Z\text{ is a minimal $H$-subflow}\}$$
		is closed in $K(X)$.
	\end{lemma}

	\begin{proof}
		We have a natural bijection between $\mathrm{Min}_H(X)$ and $\ap_H(X)/E_H$, and we will show that this is a homeomorphism.
		Since $H$ is CAP, $E_H$ is a closed equivalence relation on $\ap_H(X)$.  In general, when considering a closed equivalence relation on a compact space, the Vietoris topology may be finer than the quotient topology. So we need to show that every Vietoris open set is open in the quotient topology.

		First we consider $\mathrm{Sub}(A) \cap \mathrm{Min}_H(X)$. Let $\pi\colon \ap_H(X)\to \ap_H(X)/E_H$ be the quotient map. Then $\pi[\ap_H(X)\setminus A]\subseteq \ap_H(X)/E_H$ is quotient-topology closed, so $\pi^{-1}\left(\pi[\ap_H(X)\setminus A]\right)$ is closed, $E_{H}$-invariant, and coincides with $\mathrm{Min}_H(X)\setminus \mathrm{Sub}(A)$, showing that $\mathrm{Sub}(A) \cap \mathrm{Min}_H(X)$ is quotient-open.
		This argument works for any closed equivalence relation on a compact space.

		Now we consider $\mathrm{Meets}(A) \cap \mathrm{Min}_H(X)$; here we will need to use more specific knowledge of the situation at hand. Given a minimal $H$-flow $Z$, we have that $Z\cap A\neq \emptyset$  if and only if $Z\subseteq HA$. So $HA\cap \ap_H(X)$ is an $E_{H}$-invariant open subset of $\ap_{H}(X)$, showing that $\mathrm{Meets}(A)\cap \mathrm{Min}_H(X)$ is quotient-open.
	\end{proof}
	\vspace{0 mm}

	\begin{proof}[Proof of \Cref{Prop:CAPCoPrePreSyndSubgroup}]
		Let $X$ be a $G$-flow, towards showing that $\ap_G(X)$ is closed. Let $x_i\to x$ with each $x_i\in \ap_G(X)$. Set $Y_i = \overline{G\cdot x_{i}}$. Let $Z_i\subseteq Y_i$ be a minimal $H$-subflow; we emphasize that since $H\subseteq G$ is not assumed to be normal, we do not have $\ap_H(Y_i) = Y_i$ like we did in \Cref{Thm:CAPShortExactSeq}.

		Write $\phi_i = \phi_{Z_i}\colon \widehat{G/H}\to K(X)$.
		Fix $i$ and $y \in Y_{i} = \overline{GZ_{i}}$.
		There is a net $g_{j} \in G$ with $y \in \lim_{j} g_{j}Z_{i}$, so $\phi_{i}(\lim_{j} g_{j} H) = \lim_{j} g_{j} \phi_{i}(H) \ni y$.
		Therefore we have $Y_i = \bigcup \phi_i\left[\widehat{G/H}\right]$, so we can find $\eta_i\in \widehat{G/H}$ with $x_i\in \phi_i(\eta_i)$.
		We may assume $\eta_i\to \eta\in \widehat{G/H}$
		and, by \Cref{Lem:IfCAPthenMinClosed}, that $Z_i\to Z\in \mathrm{Min}_H(X)$. So by \Cref{Lem:JointCont}, we have $x\in \phi_Z(\eta)$.
		As above, $\bigcup \phi_Z\left[\widehat{G/H}\right] = \overline{GZ}$.
		\vspace{2 mm}

		\begin{claim}
			$\overline{GZ}$ is a minimal $G$-flow.
		\end{claim}

		\begin{proof}
			Let $u, v\in \overline{GZ}$.
			Let $\xi \in\widehat{G/H}$ with $u  \in \phi_{Z}(\xi)$.
			By pre-syndeticity of $H$, $\widehat{G/H}$ is a minimal $G$-flow, so there is a net $g_{i} \in G$ with $g_{i} \xi \to H$.
			Then $Z = \phi_{Z}(H) = \lim g_{i} \phi_{Z}(\xi)$; by letting $p \coloneqq \lim g_{i} \in \S(G)$ we have $pu\in Z$.
			Find $q\in \S(G)$ with $v\in q\circ Z$.
			Since $Z$ is $H$-minimal, $Hpu\subseteq Z$ is dense, so we can find nets $g'_i\in G$ and $h_i\in H$ with $g'_i\to q$ and $g'_i\cdot h_ig_{i}u \to v$. So in particular $v\in \overline{Gu}$.
		\end{proof}
		\vspace{2 mm}

		As $x\in \bigcup \phi_Z\left[\widehat{G/H}\right]$, this concludes the proof of \Cref{Prop:CAPCoPrePreSyndSubgroup}.
	\end{proof}

	\section{Automorphism groups of \texorpdfstring{$\omega$}{w}-homogeneous structures}
	\label{Sec:Structures}

	In this section we extend results from \cite{Bar} and \cite{Bar2} as well as the results for Polish non-archimedean groups which have the been the focus of a long line of research, initiated with \cite{KPT}.
	Results in this field link Ramsey theoretic properties of classes of finite structures to the dynamical properties of groups of automorphisms. Given a relational structure $\b K$ in some language $L$, denote by $\age(\b K)$ the class of finite structures which embed into $\b K$.
	Then $\b K$ is $\omega$-\emph{homogeneous} if any isomorphism between finite substructures of $\b K$ extends to an automorphism of $\b K$.
	A countable $\omega$-homogeneous structure is a \emph{\fr structure}.
	For any $L$-structure, the group of automorphisms of $\b K$, denoted $\aut(\b K)$ is a topological group with the topology of pointwise convergence. By expanding $L$ to add symbols for each orbit of every finite tuple in $\b{K}$, we can assume that $\b K$ is $\omega$-homogeneous.

	The age of an $\omega$-homogeneous structure contains arbitrarily large finite structures and enjoys several combinatorial properties, including the \emph{hereditary property}, the \emph{joint embedding property}, and the \emph{amalgamation property}. For the precise definitions of these properties, see \cite{KPT}.
	Families of finite structures with such properties are called \emph{\fr families}. A classical theorem of \fr is that any \fr family $\cal K$ in a countable language admits a \emph{\fr limit}, that is, a \fr structure whose age is exactly $\cal K$.

	Given two structures $\b A, \b B$, denote by $\mathrm{Emb}(\b A, \b B)$ the set of embeddings from $\b A$ to $\b B$. We write $\b A\leq \b B$ if $\emb(\b A, \b B)\neq \emptyset$. Fix an infinite structure $\b K$; given $\b A \in \age(\b K)$ and $k< \omega$, we say that $\b A$ has \emph{Ramsey degree $k$} if $k$ is least such that for all $r \ge k$ and all $\b B \in \age(\b K)$ with $\b A\leq \b B$, there is $\b C \in \age(\b K)$ with $\b B\leq \b C$ such that for each coloring $\gamma\colon\mathrm{Emb}(\b A, \b C)\to r$, there is $i \in \mathrm{Emb}(\b B, \b C)$ such that
	\[
	\left \lvert \setnew*{ \gamma(i \circ j)}{ j \in \mathrm{Emb}(\b A, \b B)} \right \rvert \le k.
	\]
	The class $\age(\b K)$ has the \emph{Ramsey property} if each $\b A \in \age(\b K)$ has Ramsey degree $1$ and has \emph{finite Ramsey degrees} if each $\b A \in \age(\b K)$ has Ramsey degree some $k < \omega$.

	For automorphism groups of countable structures, we have the following:
	\vspace{2 mm}

	\begin{fact}[\cite{ZucAut}]
		\label{Fact:ZucNonArch}
		Let $\b K$ be a \fr structure.
		Then $\M(\aut(\b K))$ is metrizable if and only if each $\b A \in \age(\b K)$ has finite Ramsey degrees.
	\end{fact}
	\vspace{2 mm}

	Fix $\b K$ be an $\omega$-homogeneous relational $L$-structure with $\cal{K} = \age(\b{K})$.
	If $L^* \supseteq L$ is a larger language, let $X_{L^*}$ denote the set of $L^*$-expansions of $\b{K}$. The group $\aut(\b{K})$ acts on $X_{L^*}$ in the obvious fashion, where if $\b{K}^*\in X_{L^*}$, $R\in L^*\setminus L$, $g\in \aut(\b{K})$, and $R^{\b{K}^*}(a_0,...,a_{n-1})$ holds, then $R^{g\cdot \b{K}^*}(ga_0,...,ga_{n-1})$ holds.
	Now suppose $\cal{K}^*$ is a hereditary family of $L^*$-structures such that $\cal{K}^* 	\vert_L = \cal{K}$. We say that $\cal{K}^*$ is a \emph{reasonable} expansion of $\cal{K}$ if whenever $\b A\subseteq \b B\in \cal{K}$ and $\b A^*\in \cal{K}^*$ is an expansion of $\b A$, then there is some expansion $\b B^*$ of $\b B$ whose restriction to $\b{A}$ is $\b{A}^*$. Given a reasonable expansion class $\cal{K}^*$, one can form the following $\aut(\b{K})$-invariant subset of $X_{L^*}$:
	$$X_{\cal{K}^*} \coloneqq \{\b{K}^*\in X_{L^*}: \age(\b{K}^*)\subseteq \cal{K}^*\}.$$
	If $\cal{K}^*$ contains finitely many expansions of each $\b A \in \age(\b K)$, we say that $\cal{K}^*$ is a pre-compact expansion of $\cal{K}$. When $\cal{K}^*$ is a reasonable, pre-compact expansion of $\cal{K}$, one can endow $X_{\cal{K}^*}$ with a natural compact topology, turning $X_{\cal{K}^*}$ into an $\aut(\b{K})$-flow. There is a combinatorial property of $\cal{K}^*$, called either \emph{expansion, order}, or \emph{minimal} property, which holds  if and only if $X_{\cal{K}^*}$ is a minimal flow. If in addition $\cal{K}^*$ is a \fr family with the Ramsey property, then $\cal{K}^*$ is called an \emph{excellent} expansion of $ \age(\b K)$. We have the following fact, which was first proved in the countable case in \cite{KPT}, and subsequently generalized in \cite{Bar}.
	\vspace{2 mm}

	\begin{fact}[\cites{Bar, KPT}]
		\label{Fact:BartUncountable}
		Let $\b K$ be a $\omega$-homogeneous structure.
		If $\age(\b K)$ admits an excellent expansion $\cal{K}^*$, then $\M(\aut(\b K)) = X_{\cal K^*}$ is the space of expansions of $\b K$.
	\end{fact}
	\vspace{2 mm}

	In \cite{Bar2} and \cite{BarMore}, Barto\v{s}ov\'{a} computes the universal minimal flows of several groups of automorphisms of uncountable $\omega$-homogeneous structures.
	One such group is the group of permutations $\sym(\kappa)$ of an arbitrary cardinal $\kappa$, whose universal minimal flow is the space $\mathrm{LO}(\kappa)$ of linear orders on $\kappa$, with the topology whose basic open sets consist of all linear orders which extend a given linear order on a finite subset.
	Other examples of UMFs computed in \cite{Bar2} and \cite{BarMore} include: the group of automorphisms of an uncountable homogeneous boolean algebra; the automorphism group of an infinite dimensional vector space over a finite field, and the group of automorphisms of any $\omega$-homogeneous graph which embeds every finite graph.

	A necessary and sufficient condition for this existence of an excellent expansion was given in \cite{ZucAut}.
	\vspace{2 mm}

	\begin{fact}[\cite{ZucAut}]
		Let $\cal K$ be a countable \fr family of finite structures.
		$\cal K$ has finite Ramsey degrees if and only if it admits an excellent expansion $\cal K^*$.
	\end{fact}
	\vspace{2 mm}

	The main theorem of this section extends \Cref{Fact:ZucNonArch} to the case of automorphism groups of uncountable $\omega$-homogeneous structures.
	\vspace{2 mm}

	\begin{theorem}
		\label{Thm:FiniteRamseyIffCAP}
		Let $\b K$ be a $\omega$-homogeneous relational structure.
		Then $\aut(\b K)$ is CAP if and only if $\age(\b K)$ has finite Ramsey degrees.
	\end{theorem}
	\vspace{2 mm}

	Fix a $\omega$-homogeneous relational structure $\b K$.
	We write $G$ for $\aut(\b K)$.
	For each $\b A \in \age(\b K)$, let $\iota_{\b A}$ denote the inclusion map $\b A \to \b K$, and $H_{\b A}$ be the collection of embeddings from $\b A$ to $\b K$, which we regard as a discrete space.
	If $\b A \subseteq \b K$ are two finite substructures of $\b K$, there is canonical projection $\pi_{\b A}^{\b B}\colon H_{\b B} \to H_{\b A}$, given by the restriction to $\b A$. Letting $\pi_{\b A}\colon G \to H_{\b A}$ denote the map $g \mapsto g^{-1} \circ \iota_{\b A}$, then $\pi_{\b A}^{\b B} \circ \pi_{\b B} = \pi_{\b A}$ for any $\b A \subseteq \b B \in \age(\b K)$. Note that if $f\in H_\b{A}$ and $g\in G$, we can interpret $f\cdot g$ as an element of $H_{g^{-1}\b{A}}$ in the natural way.

	Let $U_\b{A}$ denote the pointwise stabilizer of $\b{A}$.
	Then $\{U_\b{A}: \b{A} \in \age(\b K)\}$ is a base of clopen neighborhoods of $1_G$. We remark that $\pi_{\b A}$ is right uniformly continuous, as
	\begin{align*}
	gh^{-1} \in U_{\b A} &\Leftrightarrow \pi_{\b A}(g h^{-1}) = \iota_{\b A}\\[1 mm]
	&\Leftrightarrow  h \circ g^{-1} \circ \iota_{\b A} = \iota_{\b A}\\[1 mm]
	&\Leftrightarrow g^{-1} \circ \iota_{\b A} = h^{-1} \circ \iota_{\b A}\\[1 mm]
	&\Leftrightarrow \pi_{\b A}(g) =  \pi_{\b A}(h).
	\end{align*}
	We take the Stone-\v{C}ech compactification $\beta H_{\b A}$.
	The maps $\pi_{\b A}^{\b B}$ extend to continuous maps $\beta H_{\b B} \to \beta H_{\b A}$, under which we form the inverse limit $\varprojlim \beta H_{\b A}$. The group $G$ acts on $\varprojlim \beta H_{\b A}$ on the left as follows, where if $g\in G$ and $x\coloneqq (x_{\b{A}})_{\b{A}}\in \varprojlim \beta H_{\b{A}}$, we have $gx \coloneqq ((gx)_{\b{A}})_{\b{A}}$ given by the formula
	$$S\in (gx)_\b{A} \Leftrightarrow S\cdot g\in x_{g^{-1}\b{A}}.$$
	The following generalizes \cite{ZucAut}, see also \cite{KrPi} or \cite{Pes}.
	\vspace{2 mm}

	\begin{prop}
		\label{Prop:LevelSG}
		$\S(G)\cong \varprojlim \beta H_\b{A}$.
	\end{prop}

	\begin{proof}
		Notice that $\pi_{\b A}$ continuously extends to a surjection $\tilde \pi_{\b A} \colon \S(G) \to \beta H_{\b A}$, since $\tilde \pi_{\b A}$ is uniformly continuous.
		Thus we obtain a surjective $G$-map $\tilde \pi\colon \S(G)\to \varprojlim \beta H_\b{A}$ by setting $\tilde \pi(p) = (\tilde \pi_\b{A}(p))_{\b{A}}$.

		To see that $\tilde \pi$ is injective, suppose $p\neq q\in \S(G)$.
		Then we can find $S\in p$, $T\in q$, and $\b{A} \in \age(\b K)$ with $U_\b{A}S\cap U_\b{A}T = \emptyset$.
		So $U_{\b A} S = \pi_{\b A}^{-1}(\pi_{\b A}[S])$ belongs to $p$ but not to $q$, which implies that $\tilde \pi_\b{A}(p)\neq \tilde \pi_\b{A}(q)$.
	\end{proof}
	\vspace{2 mm}

	We refer to this as the \emph{level representation} of the Samuel compactification.
	Recall that a neighborhood basis of $p \in \S(G)$ is given by the clopen sets $C_{U_{\b A}B}$, for $\b A \in \age(\b K)$ and $B \in p$.

	A key ingredient in the proof of \Cref{Thm:FiniteRamseyIffCAP} is a lemma which links the combinatorial properties of $\age(\b K)$ to the level representation of any minimal subflow of $\S(G)$.
	\vspace{2 mm}

	\begin{lemma}
		\label{Lem:FiniteRamseyIffFiniteProj}
		Let $M \subseteq \S(G)$ be a minimal subflow.
		Then $\age(\b K)$ has finite Ramsey degrees if and only if $\tilde \pi_{\b A}[M] \subseteq \beta H_{\b A}$ is finite for all $\b A \in \age(\b K)$.
	\end{lemma}
\vspace{2 mm}

	The proof of \Cref{Lem:FiniteRamseyIffFiniteProj} is essentially that of \cite{ZucThesis}*{Theorem 3.5.5}.
	The interested reader should be aware that the side of the group actions in \cite{ZucThesis} is reversed with respect to our presentation.

	\begin{proof}[Proof of \Cref{Thm:FiniteRamseyIffCAP}]
		For each $A \in \age(\b K)$ we define a discrete pseudo-metric $d_{\b A}$ on $G$ which on input $g, h$ takes value $0$ if $\pi_{\b A}(g) = \pi_{\b A}(h)$, and $1$ overwise.
		Notice that $d_{\b A}$ is continuous and right-invariant.
		It follows that $\setnew{d_{\b A}}{A \in \age(\b K)}$ is a basis of continuous pseudo-metrics for $G$, as the $U_{\b A}$'s are a basis of neighborhoods of $1_{G}$.
		\vspace{0 mm}

		\begin{claim}
			\label{Claim:InducedPseudoAreDiscrete}
			For each $A \in \age(\b K)$, the pseudo-metric $\partial_{\b A} \coloneqq \partial_{d_{\b A}}$ on $\S(G)$ is discrete and such that $\partial_{\b A}(p, q) = 0$ if and only if $\tilde \pi_{\b A}(p) = \tilde \pi_{\b A}(q)$.
			In particular, $\partial_{\b A}$ induces a discrete metric on $\beta H_{\b A}$.
		\end{claim}

		\begin{proof}[Proof of Claim]
			Fix $\b A \in \age(\b K)$.
			For all $c <1$, we have $d_{\b A}(c+\varepsilon) = U_{\b A}$.
			By \Cref{Prop:MHPPseudo}, we therefore have that $\partial_{\b A}(p, q) < 1$ if and only if $\partial_{\b A}(p, q)= 0$ if and only if for all $B \ni_{op} p$ it holds that $q \in \overline{U_{\b A}B}$.
			On the other hand $d_{\b A}(1+\varepsilon) = G$, so $\partial_{\b A}(g, h)\neq 0$ implies $\partial_{\b A}(g, h) =1$.

			If $p, q \in \S(G)$ are such that $\tilde \pi_{\b A}(p) = \tilde \pi_{\b A}(q)$ and $B \in p$, then $\pi_{\b A}[B] \in \tilde \pi_{\b A}(p) = \tilde \pi_{\b A}(q)$, so $q \ni \pi_{\b A}^{-1}(\pi_{\b A}[B]) = U_{\b A} B$,
			that is $q \in C_{U_{\b A}B} = \overline{U_{\b A}C_{B}}$, where the last equality follows from point 4 of
			\Cref{Fact:NUltBasics}.
			By \Cref{Prop:MHPPseudo}, we have that $\partial_{\b A}(p, q) = 0$, since $C_{B}$ is a basic closed set to which $p$ belongs.

			If $\tilde \pi_{\b A}(p) \neq \tilde \pi_{\b A}(q)$, let $S \subseteq H_{\b A}$ be such that $\pi_{\b A}^{-1}(S) \in p$ but $\pi_{\b A}^{-1}(S) \not \in q$.
			Then $U_{\b A} \pi_{\b A}^{-1}(S) = \pi_{\b A}^{-1}(S) \not \in q$.
			So $q \not \in C_{U_{\b A}\pi_{\b A}^{-1}(S)} = \overline{U_{\b A}\pi_{\b A}^{-1}(S)}$. But $\pi_{\b A}^{-1}(S)$ is a neighborhood of $p$, thus
			$\partial_{\b A}(p, q) >0$, by \Cref{Prop:MHPPseudo}
		\end{proof}
\vspace{1 mm}

		Let $M \subseteq \S(G)$ be minimal.
		If we suppose that $\age(\b K)$ does not have finite Ramsey degrees, by \Cref{Lem:FiniteRamseyIffFiniteProj} there is $\b A \in \age(\b K)$ with $\pi_{\b A}[M]$ infinite.
		If $G$ were CAP, then $\partial_{\b A}$ would be continuous on $M$ by point 4 of \Cref{Thm:EquivalentInG} and thus $\partial_{\b A}$ would induce a continuous discrete metric on $\pi_{\b A}[M]$.
		This is a contradiction because $\pi_{\b A}[M]$ is compact and infinite and thus not discrete.

		On the other hand, if $\age(\b K)$ has finite Ramsey degrees, $\pi_{\b A}[M]$ is finite for each $\b A \in \age(\b K)$.
		But then $\partial_{\b A}$ is continuous as all its balls are clopen.
		Since $\setnew{d_{\b A}}{\b A \in \age(\b K)}$ is a basis of continuous pseudo-metrics, $G$ is CAP by condition 5 of \Cref{Thm:EquivalentInG}.
	\end{proof}
\vspace{2 mm}

	We finish the section by recalling a question first posed in \cite{Bar}, which is closely related to \Cref{Que:MNTConverse}.
	\vspace{2 mm}

	\begin{que}
		\label{Que:MNTConverseNonArch}
		Let $\b K$ be a $\omega$-homogeneous relational structure such that $\aut(\b K)$ is CAP.
		Does there exist an excellent expansion $\cal K^*$ of $\age(\b K)$ and $\b K^* \in X_{\cal K^*}$ such that $\b K^*$ is $\omega$-homogeneous?
		Is this the case for any excellent expansion $\age(\b K)$?
	\end{que}
\vspace{2 mm}

	The automorphism group of such $\b K^*$ would indeed be a closed co-precompact and extremely amenable subgroup of $\aut(\b K)$ such that $\M(\aut(\b K)) = X_{\cal K^*} = \widehat{\aut(\b K) / \aut(\b K^*)}$.

	The question is open even in the concrete case in which $\b K $ is an uncountable $\omega$-homogeneous graph which embeds all finite graphs.
	An excellent expansion of the class of finite graphs is the class of all finite linearly ordered graphs.
	It is not known if there is a linear order on $\b K$ such that the resulting structure is $\omega$-homogeneous.

	\subsection{The universal minimal flow of \texorpdfstring{$\homeo(\omega_1)$}{Homeo(w1)}}

	In this section we make use of our results to compute the universal minimal flow of a concrete group.
	In a recent paper \cite{Ghey}, Gheysens investigates the space $\omega_{1}$ with the order topology.
	He proves that its group of homeomorphisms, which is a topological group with the topology of pointwise convergence, is amenable, Roelcke-precompact while not being Baire, and admitting only trivial homomorphism to metrizable groups.
	The result which is of interest for the computation of its universal minimal flow is \cite{Ghey}*{Lemma 11}, whose immediate consequence is:
	\vspace{2 mm}

	\begin{fact}
		\label{Fact:HomeoSym}
		The group $\homeo(\omega_{1})$ densely embeds in $\sym(\omega_{1})^{\omega_{1}}$.
	\end{fact}
\vspace{2 mm}

	By \Cref{Fact:DenseSubgroup} universal minimal flow of a group coincides with that of any of its dense subgroups, so we are reduced to computing $\M(\sym(\omega_{1})^{\omega_{1}})$.
	By \cite{Pes2} and \Cref{Thm:FiniteRamseyIffCAP}, $\sym(\omega_1)$ is a CAP group and its universal minimal flow is the space $\mathrm{LO}(\omega_{1})$ of linear orders of $\omega_{1}$.

	Since $\sym(\omega_1)$ is CAP, so is $\sym(\omega_1)^{\omega_1}$, by \Cref{Cor:CAPProducts}.
	By \Cref{Fact:DenseSubgroup}, $\homeo(\omega_1)$ is CAP and by \Cref{Prop:ProductFormula} we have that $\M(\sym(\omega_{1})^{\omega_{1}}) = \M(\sym(\omega_{1}))^{\omega_{1}} = \mathrm{LO}(\omega_{1})^{\omega_1}$.
	We have thus proven that:
	\vspace{2 mm}

	\begin{prop}
		\label{Prop:UMFHomeoOmegaOne}
		$\homeo(\omega_1)$ is a CAP group and its universal minimal flow is the space $\mathrm{LO}(\omega_{1})^{\omega_1}$.
	\end{prop}
\vspace{2 mm}

	The result holds more generally.
	A topological space is \emph{scattered} if it contains no perfect subspaces.
	If $X$ is scattered, the topologies of pointwise convergence and of discrete pointwise convergence on $\homeo(X)$ coincide \cite{Ghey}*{Proposition 1}.
	It follows that $\homeo(X)$ embeds in $\sym(\lvert X \rvert)$.
	Let $X^{(\alpha)}$ denote the $\alpha$\emph{-th Cantor-Bendixson derivative} of $X$, and $\mathrm{CB}(X)$ be its \emph{Cantor-Bendixson rank} (see \cite{Ghey} for the definitions).
	The same reasoning as above leads to the following.
	\vspace{2 mm}

	\begin{prop}
		\label{Prop:GenScattered}
		Suppose that $X$ is a scattered space such that for any $k < \omega$ and ordinals $\alpha_{1}, \dots, \alpha_{k} < \mathrm{CB}(X)$, whenever $(x_{1}, \dots, x_{k}), (y_{1}, \dots, y_{k})$ are tuples of distinct points such that $x_{i}, y_{i} \in X^{(\alpha_{i} + 1)} \setminus X^{(\alpha_{i})}$,
	 there is $g \in \homeo(X)$ with $g(x_{i}) = y_{i}$, for $1 \le i \le k$.
		Then $\homeo(X)$ is CAP, it embeds densely in $\prod_{\alpha<\mathrm{CB(X)}} \sym\left(\left\lvert X^{(\alpha + 1)} \setminus X^{(\alpha)} \right\vert\right)$, and
		\[
		\M(\homeo(X)) = \prod_{\alpha<\mathrm{CB(X)}} \mathrm{LO}\left(\left\lvert X^{(\alpha + 1)} \setminus X^{(\alpha)} \right\vert\right).
		\]
	\end{prop}
\vspace{2 mm}

	The class of scattered spaces which satisfy the hypotheses of \Cref{Prop:GenScattered} contains all ordinals, as remarked in \cite{Ghey}.


\footnotesize
  \bigskip
  \footnotesize

  \textsc{D\'epartement des Opérations, Universit\'e de Lausanne, Quartier UNIL-Chambronne B\^atiment Anthropole, 1015 Lausanne, Switzerland}\par\nopagebreak
  \textit{Current address:} Institut Camille Jordan, Université Claude Bernard Lyon 1, Université de Lyon, 43, boulevard du 11 novembre 1918, 69622 Villeurbanne cedex, France \par\nopagebreak
 \textit{E-mail address}: \texttt{basso@math.univ-lyon1.fr}\par\nopagebreak

\bigskip
	\textsc{Department of Mathematics, UC San Diego, 9500 Gilman Drive, La Jolla, CA 92093}\par\nopagebreak
 \textit{E-mail address}: \texttt{azucker@ucsd.edu}\par\nopagebreak

\end{document}